\documentclass[a4paper,11pt,reqno]{amsart}

\usepackage{amsfonts}
\usepackage{amsmath}
\usepackage{amssymb}

\usepackage{mathrsfs}
\usepackage[colorlinks]{hyperref}
\numberwithin{equation}{section}

\usepackage{graphicx}

\newtheorem{theorem}{Theorem}[section]
\newtheorem{defi}[theorem]{Definition}
\newtheorem{remark}[theorem]{Remark}
\newtheorem{corollary}[theorem]{Corollary}
\newtheorem{prop}[theorem]{Proposition}
\newtheorem{lemma}[theorem]{Lemma}

\newenvironment{pff}{\hspace*{-\parindent}{\bf Proof \,}}
{\hfill $\Box$ \vspace*{0.2cm}}

\def\R2n{{\mathbb R}^{2n}}

\def\R2{{\mathbb R}^2}
\def\R2n{{\mathbb R}^{2n}}

\def\N0{{\mathbb N}_{0}}

\def\l2h{{\ell^2(\hbar\mathbb Z^n)}}

\begin{document}
	\title[ Weighted periodic and discrete  Pseudo-Differential Operators]{ Weighted periodic and discrete   Pseudo-Differential Operators}
	\author[Aparajita Dasgupta]{Aparajita Dasgupta}
	\address{
		Aparajita Dasgupta:
		\endgraf
		Department of Mathematics
		\endgraf
		Indian Institute of Technology, Delhi, Hauz Khas
		\endgraf
		New Delhi-110016 
		\endgraf
		India
		\endgraf
		{\it E-mail address} {\rm adasgupta@maths.iitd.ac.in}
	}
	\author[Lalit Mohan]{Lalit Mohan}
	\address{
		Lalit Mohan:
		\endgraf
		Department of Mathematics
		\endgraf
		Indian Institute of Technology, Delhi
		\endgraf
		India
		\endgraf
		{\it E-mail address} {\rm mohanlalit871@gmail.com}
	}
	\author[Shyam Swarup Mondal]{Shyam Swarup Mondal}
	\address{
		Shyam Swarup Mondal:
		\endgraf
		Department of Mathematics
		\endgraf
		Indian Institute of Technology, Delhi
		\endgraf
		India
		\endgraf
		{\it E-mail address} {\rm mondalshyam055@gmail.com}
	}
	\date{\today}

	\begin{abstract}


		In this paper, we study elements of symbolic calculus for pseudo-differential operators associated with the weighted symbol class  $M_{\rho, \Lambda}^m(\mathbb{ T}\times \mathbb{Z})$   (associated to a suitable weight function  $\Lambda$ on $\mathbb{Z}$)  by deriving formulae for the asymptotic sums, composition, adjoint, transpose. We also construct the parametrix of $M$-elliptic  pseudo-differential operators on $\mathbb{ T}$.
		Further, we prove a version of Gohberg's lemma  for pseudo-differetial operators with  weighted symbol class  $M_{\rho, \Lambda}^0(\mathbb{ T}\times \mathbb{Z})$ and  as an application, we provide a sufficient and necessary condition to  ensure that the corresponding pseudo-differential operator is compact on   $L^2(\mathbb{T})$. 
		Finally, we provide     G\r{a}rding's and Sharp  G\r{a}rding's  inequality for $M$-elliptic operators on $\mathbb{Z}$ and  $\mathbb{T}$, respectively,  and present an application  in the context of strong solution of the  pseudo-differential equation $T_{\sigma} u=f$   in $L^{2}\left(\mathbb{T}\right)$.
		
	\end{abstract}
	\maketitle
	\tableofcontents
	
	\section{Introduction}
	The theory of pseudo-differential operators plays an important role   in modern mathematics due to the fact that it has drawn a significant motivation from partial differential equations,  signal processing, and time-frequency analysis, see \cite{Hor, MR&VT book, fis1}. Pseudo-differential operators acting on functions defined on smooth manifolds are an essential generalization of differential operators. The study of pseudo-differential operators originated  in 1960s with the works of  Kohn and Nirenberg \cite{Niren} and H\"ormander  \cite{Hor} 	in the study of singular integral differential operators, mainly for inverting
	differential operators to solve elliptic differential equations. 
	Ever since the theory is a key tool,  mostly for its connections with mathematical physics and   in many areas of harmonic analysis,    quantum field theory,   and the index theory.

	In general, using the Mikjlin-H\"ormander theorem for Fourier multipliers,  pseudo-differential operators associated to    the class $S_{1, 0}^{0}$  are $L^{p}$-bounded.   It is also well known that the $L^2$-boundedness property holds for  pseudo-differential operators associated with the symbols, $S_{\rho, \delta}^{0}$  with $0 \leq \delta < \rho \leq  1$    (see \cite{Beals,Hor}). However, the situation becomes completely different for the case  $p\neq2$.
	Fefferman   \cite{feff} proved that  the pseudo-differential operators whose symbols belong to the class $S_{\rho, 0}^{0}$  with $0<\rho<1$ are not in general  $L^{p}$-bounded    for  $p \neq 2$. To avoid the difficulty described above, Taylor in \cite{Ta} introduced a suitable symbol subclass $M_{\rho,0}^{m}$, of $S_{\rho, 0}^{0}$ and developed symbolic calculus for the associated pseudo-differential operators. Further, Garello and Morando \cite{GM,mor05} introduced subclass  $M_{\rho, \Lambda}^{m}$  of $S_{\rho, \Lambda}^{m}$, which are just a weighted version of the symbol class introduced by Taylor and developed the symbolic calculus for the associated pseudo-differential operators  with many applications to  study the  regularity of multi-quasi-elliptic operators.

	Considerable attention has been devoted in the past sixteen  years to study   various properties of pseudo-differential operators associated with the symbol class $M_{\rho, \Lambda}^{m}$ on $\mathbb{R}^n$ in various directions by several researchers.   For instance,  the  symbolic calculus,  parametrix, and $L^p,1 < p < \infty,$  boundedness of      pseudo-differential  operators with  symbol    in    $M_{\rho, \Lambda}^{m}, m>0$ has  been studied by Wong    \cite{wong06}.   Further,       Kalleji   \cite{kal}  constructed  a weighted symbol class $M_{\rho, \Lambda}^m(\mathbb{ T}^n\times \mathbb{Z}^n), m\in \mathbb{R}$ associated to a suitable weight function  $\Lambda$ on $\mathbb{ Z }^n$  and study   minimal and maximal extensions,  among some other results for   pseudo-differential operators  associated  with symbol in $M_{\rho, \Lambda}^m(\mathbb{ Z}^n\times \mathbb{T}^n).$  We also note that recently, the authors in  \cite{SSM&VK}  constructed and studied $M$-elliptic pseudo-differential operators on $\mathbb{ Z }^n$ with symbol in  $M_{\rho, \Lambda}^m(\mathbb{ Z}^n\times \mathbb{T}^n)$   which is  just a weighted version of the    H\"ormander symbol class $ S_{\rho}^{m}\left(\mathbb{Z}^{n} \times \mathbb{T}^{n}\right)$ on $\mathbb{Z}^n$, introduced by Botchway, Kabiti and Ruzhansky \cite{AB&GK&MK}. More details about discrete pseudo-differential operators can be found in \cite{Cardonaaa,Kumar202020,CardonaKumar1,CardonaKumar}. This paper investigates $M$-elliptic pseudo-differential operators on $\mathbb{ T}$. In particular, we investigate elements of symbolic calculus for pseudo-differential operators associated with the  class  $M_{\rho, \Lambda}^m(\mathbb{ T}\times \mathbb{Z})$   by deriving formulae for the asymptotic sums of symbols, composition, adjoint, transpose. We also construct the parametrix of $M$-elliptic  pseudo-differential operators on $\mathbb{ T}$. e can extend this result for $\mathbb{T}^n$.

	Another important result in analysis is the Gohberg lemma due to its application in spectral theory and the singular integral equations. Gohberg lemma was first obtained by Gohberg \cite{Goh} to investigate integral operators. In 1970, the Gohberg lemma for pseudo-differential operators with bounded symbols was obtained by Gru\v{s}hin \cite{VVG}. Later, for symbols in the H\"ormander class  $S^0_{1, 0}(\mathbb{ S}^1\times \mathbb{ Z })$, an analogue of Gohberg’s lemma has been proved  in \cite{SM&MW}  to  prove the spectral invariance. Further, this theory has been extended to compact Lie group by  the first author and  Ruzhansky in \cite{dasgupta}. Recently,  for symbols in the H\"ormander class $S^0_{ 1,0}(\Omega \times \mathcal{I})$, Ruzhansky and Publo \cite{publo}   investigated a “non-harmonic version” of Gohberg’s lemma, and provided a sufficient and necessary condition to ensure that the corresponding pseudo-differential operator is a compact operator in $L^2(\Omega)$. In this manuscript, we also    establish a version of  the Gohberg lemma but for  the weighted symbol  class   $M_{\rho, \Lambda}^0(\mathbb{ T}\times \mathbb{Z})$. As an application of  Gohberg’s lemma, we also provide a sufficient and necessary condition to  ensure that the corresponding pseudo-differential operator is compact on   $L^2(\mathbb{T})$. Particularly,   we   evaluate    the norm of $T_\sigma-K$, where $K$ is a compact operator  and $\sigma\in M_{\rho, \Lambda}^0(\mathbb{ T}\times \mathbb{Z})$, and    give  estimates  for the essential spectrum of such  operator $T_\sigma$. Using the relation between the lattice quantization and the toroidal quantization developed in  \cite{AB&GK&MK}, we prove     Gohberg's lemma for  the weighted symbol  class   $M_{\rho, \Lambda}^0(\mathbb{ Z}\times \mathbb{T})$ on $\mathbb{Z}.$

	G\r{a}rding's  type inequality plays a crucial role in the study of several problems related to initial value problem of parabolic type.   Investigation of   G\r{a}rding's inequality for strongly elliptic operators was first proved by G\r{a}rding  \cite{LG}    to derive the existence of solutions of the Dirichlet problem for elliptic operators as well as to study the distribution of the eigenvalues. After that, considerable attention has been devoted by several researchers to studying G\r{a}rding's inequality for pseudo-differential operators associated with the H\"{o}rmander symbols with applications to PDE  in different contexts. For example,   G\r{a}rding's inequality for pseudo-differential operators associated with the H\"{o}rmander symbols on $\mathbb{R}^n$ with $0 \leq \delta \leq \rho \leq 1$, on compact Lie group with matrix-valued symbols, and in context of non-harmonic analysis on general smooth manifolds can be found in \cite{Ta,MR&JW,DC&MR,garding1}.    Here we would like to note that, recently, the first and second authors proved the G\r{a}rding's inequality for  SG $M$-elliptic operators to obtain results about the existence and uniqueness of solutions of the parabolic type IVP. On the other hand,  in order to deal with non-elliptic problems,  H\"ormander proved in \cite{Horman}   sharp G\r{a}rding's  inequality for operators with symbols having nonnegative real part. The sharp G\r{a}rding's inequality and its generalizations become an essential tool to investigate  the existence of solutions to a wide class of boundary value problems and to analyze the global solvability and the local well-posedness of the Cauchy problem for evolution equations. 	The sharp G\r{a}rding's inequality on $\mathbb{R}^n$ is one of the most important tools of the microlocal analysis with numerous applications in the theory of PDE \cite{Horman}.  	Notably, the sharp G\r{a}rding's inequality requires the condition   imposed on the full symbol.   Further,    the sharp G\r{a}rding's inequality for the Kohn-Nirenberg classes $S^m_{1,0}(G)$  proved in  \cite{Michael&Turunen}. Recently, the authors in \cite{sharp garding}   extended these inequality for the H\"ormander classes $S^m_{\rho,\delta}(G)$ for all $0  \leq \delta<\rho\leq  1$.   In this paper, we  prove the G\r{a}rding's  and sharp G\r{a}rding's   inequality for $M$-elliptic pseudo-differential operator with symbol in $M_{\rho, \Lambda}^0(\mathbb{ T}\times \mathbb{Z})$ on $\mathbb{T}$ and $M_{\rho, \Lambda}^0(\mathbb{ Z}\times \mathbb{T})$ on $\mathbb{ Z }$, respectively.    We also  present an application  of G\r{a}rding's   inequality  in the context of strong solution of the  pseudo-differential equation $T_{\sigma} u=f$   in $L^{2}\left(\mathbb{T}\right)$.

	The presentation of this manuscript is divided into six sections, including the introduction as follows:
	\begin{itemize}
		\item In Section \ref{sec2} we first recall some of the basics of Fourier analysis and important properties of periodic pseudo-differential operators on $\mathbb{ T}$. We also recall the weighted symbol class $M_{\rho, \Lambda}^{m}(\mathbb{T} \times \mathbb{Z} ), m\in \mathbb{R}$  associated to a suitable weight function  $\Lambda$ on $\mathbb{ Z }$ from \cite{kal}.

		\item In Section \ref{sec3} we study elements of   symbolic calculus for pseudo-differential operators associated with symbol in the weighted class  $M_{\rho, \Lambda}^m(\mathbb{ Z}\times \mathbb{T})$  by deriving formulae for the asymptotic sums, composition, adjoint, transpose. By recalling the definition of   $M$-ellipticity for symbols we construct the parametrix of $M$-elliptic  pseudo-differential operators.
		
		\item  In Section \ref{sec4} we study compact   $M$-elliptic pseudo-differential operators on $\mathbb{ T}$.  We prove a version of Gohberg's lemma  for pseudo-differential operators on $\mathbb{ T}$ and $\mathbb{ Z}$ with  symbol in the weighted symbol class  $M_{\rho, \Lambda}^0(\mathbb{ T}\times \mathbb{Z})$ and $M_{\rho, \Lambda}^0(\mathbb{ Z}\times \mathbb{T})$, respectively. Further, we provide a sufficient and necessary condition to  ensure  the compactness  of  a  pseudo-differential operator  on   $L^2(\mathbb{T})$ (also on $\ell^2(\mathbb{Z})$) with symbol in $M_{\rho, \Lambda}^0(\mathbb{ T}\times \mathbb{Z})$ (respectively in $M_{\rho, \Lambda}^0(\mathbb{ Z}\times \mathbb{T})$).
		
		\item 	 In Section \ref{sec5}    we prove   G\r{a}rding's and Sharp  G\r{a}rding's  inequality for $M$-elliptic operators on $\mathbb{Z}$ and  $\mathbb{T}$, respectively.    
		
		\item  In Section \ref{sec6}   we   discuss   an application  of G\r{a}rding's inequality  in the context of strong solution of the  pseudo-differential equation $T_{\sigma} u=f$   in $L^{2}\left(\mathbb{T}\right)$.
	\end{itemize} 

	Note that, these results can be extended easily from $\mathbb{T}$ to the $n$-dimensional torus $\mathbb{T}^{n}$ given by
	$
	\mathbb{T}^{n}=\underbrace{\mathbb{T}\times \cdots \times \mathbb{T}}_{n \text { times }},$ and similarly from $\mathbb{Z}$ to $\mathbb{Z}^{n}.$
	\section{Preliminaries}\label{sec2}
	In this section,  we first recall some notation and basic properties of periodic Fourier analysis and pseudo-differential operators on $\mathbb{T}$. We also recall   the toroidal symbol class $S_{\rho, \Lambda}^{m}\left(\mathbb{T} \times \mathbb{Z}\right)$  in the view  of Ruzhansky-Turunen theory  \cite{MR&VT book} as well as the weighted  symbol class $M_{\rho, \Lambda}^m(\mathbb{ T}\times \mathbb{Z})$ from \cite{kal}.  	We refer \cite{ MR&VT book,  MP,  SM&MW, kal, Vel} for more details and the study of various operator theoretical   properties of pseudo-differential operators  on $\mathbb{ T}.$

	The  Fourier transform $\hat{f}$ of a function $f \in L^{1}\left(\mathbb{T}\right)$ is defined by
	$$
	\widehat{f}(k)=\int_{x \in \mathbb{T}} e^{2 \pi i k \cdot x} f(x) \;dx, \quad k\in \mathbb{Z},
	$$
	where $dx$  is  the normalized Haar measure on $ \mathbb{T}$.  The above periodic  Fourier transform can be extended to $L^{2}\left(\mathbb{T}\right)$ using
	the standard density arguments. We normalize the Haar measures on $\mathbb{T}$ in such a
	manner so that the following Plancherel formula  holds:
	$$
	\int_{\mathbb{T}}|\widehat{f}(x)|^{2} d x=	\sum_{k \in \mathbb{Z}}|f(k)|^{2}.
	$$
	The inverse of the periodic Fourier transform is given by
	$$
	f(x)=\sum_{k\in \mathbb{Z}} e^{-2 \pi i k \cdot x} \widehat{f}(k), \quad x \in \mathbb{T},
	$$ where $f$ belongs to a suitable function space, namely, the Schwartz space of $\mathbb{ Z}$, $\mathcal{S}(\mathbb{ Z}),$ the space of rapidly decaying functions from  $\mathbb{ Z}\to \mathbb{C}$. 
	
	\begin{defi}(\textbf{Forward and backward differences $\triangle_{k}$ and $\bar{\triangle}_{k}$} )\\
		Let $\sigma: \mathbb{Z} \rightarrow \mathbb{C}$. We define the forward and backward partial difference operators $\triangle_{k}$ and $\bar{\triangle}_{k}$, respectively, by
		$$\triangle_{k} \, \sigma(k):=\sigma\left(k+1\right)-\sigma(k),$$
		$$\bar{\triangle}_{k} \, \sigma(k):=\sigma(k)-\sigma\left(k-1\right).$$
	\end{defi} 
	
	Let us now recall the H\"ormander symbol class, $ S_{\rho}^{m}\left(\mathbb{T}\times \mathbb{Z}\right)$, on $\mathbb{T}$, which is same as defined in \cite{MR&VT book}.  
	\begin{defi} \label{eq10} Let $m \in \mathbb{R}$ and $\rho>0.$ We say that a function $\sigma: \mathbb{T} \times \mathbb{Z} \rightarrow \mathbb{C}$ belongs to $ S_{\rho}^{m}\left(\mathbb{T} \times \mathbb{Z}\right)$  if $\sigma(x, k) $ is smooth in $x$ for all $k \in \mathbb{Z},$ and for
		all  $\alpha, \beta \in  \mathbb{N}_0$, there exists a positive constant $C_{\alpha, \beta}$ such that  
		$$
		\left|\Delta_{k}^{\alpha} \partial_{x}^{\beta} \sigma(x, k)\right| \leq C_{\alpha, \beta}(1+|k|)^{m-\rho|\alpha|}, \quad (x, k)\in  \mathbb{T} \times \mathbb{Z}.
		$$
	\end{defi}
	For the symbol $\sigma\in S_{\rho}^{m}\left(\mathbb{T}\times \mathbb{Z}\right)$, 	the corresponding pseudo-differential  operator associated with $\sigma$ is given by
	$$
	(T_\sigma f) (x):=\sum_{k\in \mathbb{Z}} e^{-2 \pi i k \cdot x}  \sigma( x, k) \widehat{f}(k) , \quad x\in \mathbb{T}.
	$$
	We denote  $\mathrm{OP} S_\rho^{m}\left(\mathbb{T} \times \mathbb{Z}\right)$ be the set of all operators  corresponding to the symbol class $S_\rho^{m}\left(\mathbb{T}\times \mathbb{Z}\right)$.

	\begin{defi}\label{weight fun. defi}
		Let $\Lambda$ be a positive function. We say that $\Lambda$ is a weight function if there exist suitable positive constants, $\mu_{0} \leq \mu_{1}$ and $C_{0}, C_{1}$ such that
		$$
		C_{0}(1+|k|)^{\mu_{0}} \leq \Lambda(k) \leq C_{1}(1+|k|)^{\mu_{1}}, \quad k \in \mathbb{Z}.
		$$
	\end{defi}
	Furthermore, we assume that there exists a real constant $\mu$ such that $\mu \geq \mu_1$ and for all $\alpha$,$\gamma$ $\in \mathbb{N}_{0}$ with $\gamma \in \{0,1\}$, we can find a positive constant $C_{\alpha,\gamma}$ such that
	\begin{eqnarray}\label{weight estimate}
		\left|k^{\gamma} \Delta_{k}^{\alpha+\gamma} \Lambda(k)\right| \leq C_{\alpha, \gamma} \Lambda(k)^{1-\frac{1}{\mu} \alpha}, \quad k \in \mathbb{Z}.
	\end{eqnarray}
	\begin{defi}\label{S symbol defi}
		Let $m \in \mathbb{R}$ and $\rho \in\left(0, \frac{1}{\mu}\right]$. Then the toroidal symbol class $S_{\rho, \Lambda}^{m}\left(\mathbb{T} \times \mathbb{Z}\right)$ is the set of all functions $\sigma: \mathbb{T} \times \mathbb{Z} \rightarrow \mathbb{C}$ which are smooth in $x$ for all $k \in \mathbb{Z}$, and for all $\alpha, \beta \in \mathbb{N}_{0}$, there exists a positive constant $C_{\alpha, \beta}$ such that
		\begin{eqnarray}\label{S class estimate}
			\left|\Delta_{k}^{\alpha} \partial_{x}^{\beta} \sigma(x, k)\right| \leq C_{\alpha, \beta} \Lambda(k)^{m-\rho \alpha}, \quad(x, k) \in \mathbb{T} \times \mathbb{Z}.
		\end{eqnarray}
	\end{defi}
	As usual we set
	$$
	S_{\rho, \Lambda}^{\infty}\left(\mathbb{T} \times \mathbb{Z}\right):=\bigcup_{m \in \mathbb{R}} S_{\rho, \Lambda}^{m}\left(\mathbb{T} \times \mathbb{Z}\right)
	$$
	and
	$$
	S_{\rho, \Lambda}^{-\infty}\left(\mathbb{T} \times \mathbb{Z}\right):=\bigcap_{m \in \mathbb{R}} S_{\rho, \Lambda}^{m}\left(\mathbb{T} \times \mathbb{Z}\right) \text {. }
	$$
	Let $\sigma \in S_{\rho,\Lambda}^{m}\left(\mathbb{T} \times \mathbb{Z}\right)$. Define a pseudo-differential operator $T_{\sigma}$ associated with symbol $\sigma$ by
	\begin{eqnarray}\label{pseudo-defi}
		T_{\sigma} f(x)= (2 \pi)^{-1} \sum_{k \in \mathbb{Z}} \int_{\mathbb{T}} e^{ i (x-y) \cdot k} \sigma(x, k) f(y) d y, \quad x \in \mathbb{T}
	\end{eqnarray}
	for every $f \in C^{\infty}\left(\mathbb{T}\right)$. We write $\operatorname{Op} (S_{\rho, \Lambda}^{m}(\mathbb{T} \times \mathbb{Z}))$ for the class of pseudo-differential operators associated with the symbol class $S_{\rho, \Lambda}^{m}\left(\mathbb{T} \times \mathbb{Z}\right)$.
	
	Now, we will describe the main ingredient of this paper, namely the symbol class $M_{\rho, \Lambda}^{m}\left(\mathbb{T} \times \mathbb{Z}\right)$.
	
	\begin{defi}\label{M class defi}
		For $m \in \mathbb{R}$ and $\rho \in\left(0, \frac{1}{\mu}\right]$, the symbol class, $M_{\rho, \Lambda}^{m}\left(\mathbb{T} \times \mathbb{Z}\right)$ consists of all functions $\sigma: \mathbb{T} \times \mathbb{Z} \rightarrow \mathbb{C}$ which are smooth in $x$ for all $k \in \mathbb{Z}$, and for $\gamma \in \{0,1\},$
		$$
		k^{\gamma} \Delta_{k}^{\gamma} \sigma(x, k) \in S_{\rho, \Lambda}^{m}\left(\mathbb{T} \times \mathbb{Z}\right).
		$$
	\end{defi} 
	In the same manner, we write $\operatorname{Op} (M_{\rho, \Lambda}^{m}(\mathbb{T} \times \mathbb{Z}))$ for the class of pseudo-differential operators associated with the symbol class $M_{\rho, \Lambda}^{m}\left(\mathbb{T} \times \mathbb{Z}\right)$.
	\begin{remark}\label{contained class}
		For every $m \in \mathbb{R}$ and $\rho \in\left(0, \frac{1}{\mu}\right]$, we have
		$$
		M_{\rho, \Lambda}^{m}\left(\mathbb{T} \times \mathbb{Z}\right) \subset S_{\rho, \Lambda}^{m}
		\left(\mathbb{T} \times \mathbb{Z}\right).
		$$
	\end{remark}
	\begin{lemma}\label{relation in both classes}
		For every $m \in \mathbb{R}$ and $0<\rho \leq \frac{1}{\mu}$, there exists a positive integer $N_{0}$ such that
		\begin{eqnarray}\label{S class in M class}
			S_{\rho, \Lambda}^{m-N_{0}}\left(\mathbb{T} \times \mathbb{Z}\right) \subset M_{\rho, \Lambda}^{m}\left(\mathbb{T} \times \mathbb{Z}\right) \subset S_{\rho, \Lambda}^{m}\left(\mathbb{T} \times \mathbb{Z}\right).
		\end{eqnarray}
		More precisely, $N_{0}:=\left(\frac{1}{\mu_{0}}-\rho\right)$. Moreover,
		\begin{eqnarray}\label{intersection of both classes}
			\bigcap_{m \in \mathbb{R}} M_{\rho, \Lambda}^{m}\left(\mathbb{T} \times \mathbb{Z}\right)=\bigcap_{m \in \mathbb{R}} S_{\rho, \Lambda}^{m}\left(\mathbb{T} \times \mathbb{Z}\right)=S_{\rho, \Lambda}^{-\infty}\left(\mathbb{T} \times \mathbb{Z}\right).
		\end{eqnarray}
	\end{lemma}
	\begin{pff}
		The proof of the above lemma is similar to the proof of Lemma 3.10 in \cite{SSM&VK}.
	\end{pff}
	\section{Symbolic calculus and parametrix for   $M_{\rho, \Lambda}^{m}(\mathbb{T} \times \mathbb{Z})$ }\label{sec3}
	In  this section  we study the symbolic calculus for pseudo-differential operators associated with symbol in  $M_{\rho, \Lambda}^m(\mathbb{ Z}^n\times \mathbb{T}^n)$  by deriving formulae for   composition, adjoint, transpose of the operators.  We also construct the parametrix of $M$-elliptic  pseudo-differential operators. We start this section with the following result related to the  asymptotic sums of symbols.

	\begin{theorem}\label{asymptotic exp.}
		Let $\left\{m_{j}\right\}_{j \in \mathbb{N}_{0}}$ be a strictly decreasing sequence of real numbers such that $m_{j} \rightarrow-\infty$ as $j \rightarrow \infty$. Suppose $ \sigma_{j} \in M_{\rho, \Lambda}^{m_{j}}\left(\mathbb{T} \times \mathbb{Z}\right), j \in \mathbb{N}_{0}$. Then there exists a symbol $\sigma \in M_{\rho, \Lambda}^{m_{0}}\left(\mathbb{T} \times \mathbb{Z}\right)$ such that
		$$
		\sigma(x, k) \sim \sum_{j=0}^{\infty} \sigma_{j}(x,k),
		$$
		i.e.,
		$$
		\sigma(x, k)-\sum_{j=0}^{N-1} \sigma_{j}(x,k) \in M_{\rho, \Lambda}^{m_{N}}\left(\mathbb{T} \times \mathbb{Z}\right),
		$$
		for every positive integer $N$.
	\end{theorem}
	\begin{pff}
		Let $\sigma_{j} \in M_{\rho, \Lambda}^{m_{j}}\left(\mathbb{T} \times \mathbb{Z}\right)$. Then from Remark \ref{contained class}, we have $\sigma_{j} \in S_{\rho, \Lambda}^{m_{j}}\left(\mathbb{T} \times \mathbb{Z}\right)$. Consider $\psi \in C^{\infty}\left(\mathbb{R}\right)$ such that $0 \leq \psi \leq 1$ and
		$$
		\psi(k)= \begin{cases}1, & \text { if }|k| \geq 1 \\ 0, & \text { if }|k| \leq \frac{1}{2}\end{cases}  .
		$$
		Let $\left(\epsilon_{j}\right)_{j=0}^{\infty}$ be a sequence of positive real numbers such that $\epsilon_{j}>\epsilon_{j+1} \rightarrow 0$. Define $\psi_{j} \in C^{\infty}\left(\mathbb{R}\right)$, by $\psi_{j}(k):=\psi\left(\epsilon_{j} k\right)$. It is clear that if $\alpha \geq 1$, then the support of $\Delta_{k}^{\alpha} \phi_{j}$ is bounded. Since $\sigma_{j} \in S_{\rho, \Lambda}^{m_{j}}\left(\mathbb{T} \times \mathbb{Z}\right)$, so using discrete Leibniz formula, we have
		$$
		\left|\Delta_{k}^{\alpha}\partial_{x}^{\beta} \left(\psi_{j}(k) \sigma_{j}(x,k)\right)\right| \leq C_{j \alpha \beta} \Lambda(k)^{m_{j}-\rho \alpha} \text {, }
		$$
		where $C_{j \alpha \beta}$ is a positive constant. This means that, $\psi_{j}(k) \sigma_{j}(x,k) \in S_{\rho, \Lambda}^{m_{j}}\left(\mathbb{T} \times \mathbb{Z}\right).$ Note that, when $j$ is large enough, $\Delta_{k}^{\alpha}\left(\psi_{j}(k) \sigma_{j}(x,k)\right)$, (where $\alpha \in \mathbb{N}_{0}$), vanishes for any fixed $k \in \mathbb{Z}$. This justifies the definition
		$$
		\sigma(x,k):=\sum_{j=0}^{\infty} \psi_{j}(k) \sigma_{j}(x,k), \quad (x,k) \in \mathbb{T} \times \mathbb{Z}.
		$$
		Clearly, $\sigma \in S_{\rho, \Lambda}^{m_{0}}\left(\mathbb{T} \times \mathbb{Z}\right)$. Further, we have
		$$
		\begin{aligned}
			&\left|\Delta_{k}^{\alpha} \partial_{x}^{\beta}\left(\sigma(x,k)-\sum_{j=0}^{N-1} \sigma_{j}(x,k)\right)\right| \\
			&\leq \sum_{j=0}^{N-1}\left|\Delta_{k}^{\alpha}\partial_{x}^{\beta}\left\{\left(\psi_{j}(k)-1\right) \sigma_{j}(x,k)\right\}\right|+\sum_{j=N}^{\infty}\left| \Delta_{k}^{\alpha}\partial_{x}^{\beta}\left(\psi_{j}(k) \sigma_{j}(x,k)\right)\right| .
		\end{aligned}
		$$
		Since $\epsilon_{j}>\epsilon_{j+1}$ and $\epsilon_{j} \rightarrow 0$ as $j \rightarrow \infty$, so $\sum_{j=0}^{N-1}\left| \Delta_{k}^{\alpha}\partial_{x}^{\beta}\left\{(\psi_{j}(k)-1) \sigma_{j}(x,k)\right\}\right|$ vanishes, whenever $|k|$ is large. Thus, there exists a positive constant $C_{r N \alpha \beta}$ such that
		$$
		\sum_{j=0}^{N-1}\left| \Delta_{k}^{\alpha}\partial_{x}^{\beta}\left\{\left(\psi_{j}(k)-1\right) \sigma_{j}(x,k)\right\}\right| \leq C_{r N \alpha \beta} \Lambda(k)^{-r}
		$$
		for any $r \in \mathbb{R}$. On the other hand, one can easily show that
		$$
		\sum_{j=N}^{\infty}\left| \Delta_{k}^{\alpha}\partial_{x}^{\beta}\left(\psi_{j}(k) \sigma_{j}(x,k)\right)\right| \leq C_{N \alpha \beta}^{\prime} \Lambda(k)^{m_{N}-\rho \alpha},
		$$
		where $C_{N \alpha \beta}^{\prime}$ is a positive constant. This shows that for every $N \in \mathbb{N}$, we have
		$$
		\sigma(x,k)-\sum_{j=0}^{N-1} \sigma_{j}(x,k) \in S_{\rho, \Lambda}^{m_{N}}\left(\mathbb{T} \times \mathbb{Z}\right) .
		$$
		Since $m_{j} \rightarrow-\infty$, as $j \rightarrow \infty$, using left inclusions in \eqref{S class in M class}, we have $\sigma-\sum_{j=0}^{N-1} \sigma_{j} \in$ $S_{\rho, \Lambda}^{m_{N}}\left(\mathbb{T} \times \mathbb{Z}\right) \subset M_{\rho, \Lambda}^{m_{0}}\left(\mathbb{T} \times \mathbb{Z}\right)$ for a sufficiently large $N$. Hence $\sigma(x,k) \in M_{\rho, \Lambda}^{m_{0}}\left(\mathbb{T} \times \mathbb{Z}\right)$. Furthermore, for all $N \geq 2$ and $N^{\prime} > N$
		$$
		\sigma-\sum_{j=0}^{N-1} \sigma_{j}=\sum_{j=N}^{N^{\prime}-1} \sigma_{j}+r_{N^{\prime}}
		$$
		with $r_{N^{\prime}} \in S_{\rho, \Lambda}^{m_{N^{\prime}}}\left(\mathbb{T} \times \mathbb{Z}\right)$. By choosing a sufficiently large $N^{\prime}$ so that $m_{N^{\prime}}<m_{N}-$ $N_{0}$, we have $r_{N^{\prime}} \in S_{\rho, \Lambda}^{m_{N}-N_{0}}\left(\mathbb{T} \times \mathbb{Z}\right) \subset M_{\rho, \Lambda}^{m_{N}}\left(\mathbb{T} \times \mathbb{Z}\right)$ and therefore $\sigma-\sum_{j=0}^{N-1} \sigma_{j} \in$ $M_{\rho, \Lambda}^{m_{N}}\left(\mathbb{T} \times \mathbb{Z}\right)$. This completes the proof of the theorem.
	\end{pff}
	
	The following results on the basic symbolic calculus of pseudo-differential operators with weighted $M$-symbols on $\mathbb{T} \times \mathbb{Z}$ are analogs of results for pseudo-differential operators with symbols in $S^{m}\left(\mathbb{T}^n \times \mathbb{Z}^n \right)$ given in \cite{MR&VT book} and symbols in $M_{\rho,\Lambda}^{m}\left(\mathbb{Z}^n \times \mathbb{T}^n \right)$ given in \cite{SSM&VK}.
	\begin{theorem}\label{product}
		Let $\sigma \in M_{\rho,\Lambda}^{m}\left(\mathbb{T} \times \mathbb{Z}\right)$ and $\tau \in M_{\rho,\Lambda}^{\mu}\left(\mathbb{T} \times \mathbb{Z}\right)$. Then $T_\sigma T_\tau = T_\lambda$, where $\lambda \in M_{\rho,\Lambda}^{m+\mu}\left(\mathbb{T} \times \mathbb{Z}\right)$ and
		$$\lambda \sim \sum_{\alpha} \frac{(-i)^{|\alpha|}}{\alpha!} (\Delta_{k}^{\alpha}\sigma) (\partial_{x}^{\alpha}\tau).$$
		Here the asymptotic expansion means that
		$$\lambda - \sum_{|\alpha|<N} \frac{(-i)^{|\alpha|}}{\alpha!} (\Delta_{k}^{\alpha}\sigma) (\partial_{x}^{\alpha}\tau) \in M_{\rho,\Lambda}^{m+\mu-\rho N}\left(\mathbb{T} \times \mathbb{Z}\right),$$
		for every positive integer $N$.
	\end{theorem}
	\begin{theorem}\label{adjoint}
		Let $\sigma \in M_{\rho,\Lambda}^{m}\left(\mathbb{T} \times \mathbb{Z}\right)$. Then the formal adjoint $T_{\sigma}^{\ast}$ of  $T_{\sigma}$ is the  pseudo-differential operator  $T_{\tau}$, where $\tau \in M_{\rho,\Lambda}^{m}$ and 
		$$\tau \sim \sum_{\alpha} \frac{(-i)^{|\alpha|}}{\alpha!} \Delta_{k}^{\alpha}\partial_{x}^{\alpha}\overline{\sigma}.$$
		Here the asymptotic expansion means that
		$$\tau - \sum_{|\alpha|<N} \frac{(-i)^{|\alpha|}}{\alpha!}\Delta_{k}^{\alpha}\partial_{x}^{\alpha}\overline{\sigma} \in M_{\rho,\Lambda}^{m-\rho N}\left(\mathbb{T} \times \mathbb{Z}\right),$$
		for every positive integer $N$.
	\end{theorem}
	Let $\sigma \in M_{\rho,\Lambda}^{m}\left(\mathbb{T} \times \mathbb{Z}\right)$, where $m \in \mathbb{R}$. Then $\sigma$ is said to be $M$-elliptic if there exist positive constants $C$ and $R$ such that
	$$|\sigma(x,k)| \geq C\Lambda(k)^{m}, \quad \forall x \in \mathbb{T}, k \in \mathbb{Z},$$
	with $|k| \geq R$. Naturally, a pseudo-differential operator $T_{\sigma}$ corresponding to such $\sigma$, is said to be $M$-elliptic.\\
	The following lemma is analogous to the Lemma 2 in \cite{GM}, which can be proved in the similar way. So we skip the proof here.
	\begin{lemma}\label{inverse of symbol class}
		Let $\sigma(x,k) \in M_{\rho, \Lambda}^{m}(\mathbb{T} \times \mathbb{Z})$ and $\psi(x, k) \in C^{\infty}\left(\mathbb{T} \times \mathbb{R}\right)$, be such that there exist two sufficiently large positive constants $R^{\prime \prime}>R^{\prime}$ so that $\psi(x, k)=0$, for all $x \in \mathbb{T}$ and $|k| \leq R^{\prime}$, and $\psi(x, k)=1$, for all $x \in \mathbb{T}$ and $|k| \geq R^{\prime \prime}$. Then $q(x,k)=\frac{\psi(x, k)}{p(x, k)} \in  M_{\rho, \Lambda}^{-m}(\mathbb{T} \times \mathbb{Z})$.
	\end{lemma}
	Using the above lemma, we obtain the parametrix of the elliptic operators.
	\begin{theorem}\label{parametrix thm}
		A symbol $\sigma$ is elliptic in $M_{\rho,\Lambda}^{m}\left(\mathbb{T} \times \mathbb{Z}\right)$ if and only if there exists a symbol $\tau$ in $M_{\rho,\Lambda}^{-m}\left(\mathbb{T} \times \mathbb{Z}\right)$ such that
		\begin{eqnarray}\label{left side parametrix}
			T_{\tau} T_{\sigma}=I+R
		\end{eqnarray}
		and
		\begin{eqnarray}\label{right side parametrix}
			T_{\sigma} T_{\tau}=I+S,
		\end{eqnarray}
		where $R$ and $S$ are pseudo-differential operators with symbols in $\bigcap_{m \in \mathbb{R}} M_{\rho,\Lambda}^{m}\left(\mathbb{T} \times \mathbb{Z}\right)$, and $I$ is the identity operator.
	\end{theorem}
	\begin{pff}
		First suppose that there exists a symbol $\tau$ in $M_{\rho,\Lambda}^{-m}\left(\mathbb{T} \times \mathbb{Z}\right)$ such that \eqref{left side parametrix} and \eqref{right side parametrix} are true.
		From \eqref{right side parametrix}, we can conclude that
		$$
		I - T_{\sigma} T_{ \tau}  \in \operatorname{Op} (M_{\rho,\Lambda}^{-\infty}(\mathbb{T} \times \mathbb{Z})).
		$$
		Hence, by Theorem \ref{product}, we have
		$$
		1-\sigma(x,k) \tau(x,k) \in M_{\rho,\Lambda}^{-\rho}\left(\mathbb{T} \times \mathbb{Z}\right),
		$$
		So, we can find two positive constants $C$ and $C_{0}$ such that
		$$
		\left|1-\sigma(x,k) \tau(x,k)\right| \leq C_{0} \Lambda(k)^{-\rho} \leq C(1+|k|)^{-\rho \mu_{0}}, \quad (x,k) \in \mathbb{T} \times \mathbb{Z}.
		$$
		Let $R \in \mathbb{N}$ such that $C(1+R)^{-\rho \mu_{0}}<\frac{1}{2}$. Then it follows that
		$$
		\left|\sigma(x,k) \tau(x,k)\right| \geq \frac{1}{2}, \quad \forall \,\, |k| \geq R,
		$$
		and hence
		$$
		\left|\sigma(x, k)\right| \geq \frac{1}{2\left|\tau(x, k)\right|} \geq \frac{1}{2 C^{\prime}}(\Lambda(k))^{m}, \quad \forall \,\, |k| \geq R,
		$$
		since
		$$
		\left|\tau(x, k)\right| \leq C^{\prime}(\Lambda(k))^{-m}, \quad (x,k) \in \mathbb{T} \times \mathbb{Z}.
		$$
		Hence $\sigma$ is an $M$-elliptic symbol of order $m$.\\
		
		Conversely, let us assume that $\sigma$ is an $M$-elliptic symbol of order $m$. Then there exist positive constants $C$ and $R$ such that
		$$|\sigma(x,k)| \geq C\Lambda(k)^{m},$$
		for all $x \in \mathbb{T}$ and for all $k \in \mathbb{Z}$ with $|k| \geq R$. The idea is to find a sequence of symbols $\tau_{j} \in M_{\rho,\Lambda}^{-m-\rho j}( \mathbb{T} \times \mathbb{Z}), j=0,1,2, \ldots.$ Let us assume that this can been done. Then, by Theorem \ref{asymptotic exp.}, there exists a symbol $\tau \in M_{\rho,\Lambda}^{-m}(\mathbb{T} \times \mathbb{Z})$ such that $\tau \sim \sum_{j=0}^{\infty} \tau_{j}$, and, by Theorem \ref{product}, the symbol $\lambda$ of the product $T_{\tau} T_{\sigma}$ is in $M_{\rho,\Lambda}^{0}(\mathbb{T} \times \mathbb{Z})$ such that
		\begin{eqnarray}\label{product formula on lambda}
			\lambda-\sum_{|\gamma|<N} \frac{(-i)^{|\gamma|}}{\gamma !}\left(\partial_{x}^{\gamma} \sigma\right)(\Delta_{k}^{\gamma} \tau) \in M_{\rho,\Lambda}^{-\rho N}(\mathbb{T} \times \mathbb{Z}),
		\end{eqnarray}
		for every positive integer $N$. Also $\tau \sim \sum_{j=0}^{\infty} \tau_{j}$ implies that
		\begin{eqnarray}\label{ass. exp. on tau}
			\tau-\sum_{j=0}^{N-1} \tau_{j} \in M_{\rho,\Lambda}^{-m-\rho N}(\mathbb{T} \times \mathbb{Z}),
		\end{eqnarray}
		for every positive integer $N$. Hence, by \eqref{product formula on lambda} and \eqref{ass. exp. on tau},
		\begin{eqnarray}\label{lambda minus forward difference tau class}
			\lambda-\sum_{|\gamma|<N} \frac{(-i)^{|\gamma|}}{\gamma !}\left(\partial_{x}^{\gamma} \sigma\right) \sum_{j=0}^{N-1}(\Delta_{k}^{\gamma} \tau_{j}) \in M_{\rho,\Lambda}^{-\rho N}(\mathbb{T} \times \mathbb{Z}),
		\end{eqnarray}
		for every positive integer $N$. But we can write
		\begin{equation}\label{separation of sum}
			\begin{aligned}[b]
				& \sum_{|\gamma|<N} \frac{(-i)^{|\gamma|}}{\gamma !} \sum_{j=0}^{N-1}(\Delta_{k}^{\gamma} \tau_{j})\left(\partial_{x}^{\gamma} \sigma\right) \\
				=& \,\tau_{0} \sigma+\sum_{l=1}^{N-1}\left\{\tau_{l} \sigma+\sum_{\substack{|\gamma|+j=l \\
						j<l}} \frac{(-i)^{|\gamma|}}{\gamma !}(\Delta_{k}^{\gamma} \tau_{j})\left(\partial_{x}^{\gamma} \sigma\right)\right\} \\
				+& \sum_{\substack{|\gamma|+j \geq N \\
						|\gamma|<N,\, j<N}} \frac{(-i)^{|\gamma|}}{\gamma !}(\Delta_{k}^{\gamma} \tau_{j})\left(\partial_{x}^{\gamma} \sigma\right) .
			\end{aligned}
		\end{equation}
		To find a sequence of symbols $\tau_{j} \in M_{\rho,\Lambda}^{-m-\rho j}(\mathbb{T} \times \mathbb{Z}), j=0,1,2, \ldots$, we choose $\psi$ to be any function in $C^{\infty}\left(\mathbb{R}^{n}\right)$ such that $\psi(k)=1$, for $|k| \geq 2 R$ and $\psi(k)=0$, for $|k| \leq R$. Define
		\begin{eqnarray}\label{tau_0 defi}
			\tau_{0}(x, k)= \begin{cases}\frac{\psi(k)}{\sigma(x, k)}, & |k|>R, \\ 0, & |k| \leq R,\end{cases}, \quad (x,k) \in \mathbb{T} \times \mathbb{Z}\text{.}
		\end{eqnarray}
		From Lemma \ref{inverse of symbol class}, it is clear that $\tau_0 \in M_{\rho, \Lambda}^{-m}(\mathbb{T} \times \mathbb{Z})$. Now define, $\tau_{l}$, for $l \geq 1$, inductively by
		\begin{eqnarray}\label{tau_l defi}
			\tau_{l}=-\left\{\sum_{\substack{|\gamma|+j=l \\ j<l}} \frac{(-i)^{|\gamma|}}{\gamma !}\left(\partial_{x}^{\gamma} \sigma\right)(\Delta_{k}^{\gamma} \tau_{j})\right\} \tau_{0}.
		\end{eqnarray}
		Then it can be shown that $\tau_{j} \in M_{\rho,\Lambda}^{-m-\rho j}( \mathbb{T} \times \mathbb{Z}), j=0,1,2, \ldots$. Now, by \eqref{tau_0 defi}, $\tau_{0} \sigma=1$, for $|k| \geq 2 R$. The second term on the right hand side of \eqref{separation of sum} vanishes for $|k| \geq 2 R$ by \eqref{tau_0 defi} and \eqref{tau_l defi}. Also the third term there,
		$$
		(\Delta_{k}^{\gamma} \tau_{j})\left(\partial_{x}^{\gamma} \sigma\right) \in M_{\rho,\Lambda}^{-\rho N}(\mathbb{T} \times \mathbb{Z}),
		$$
		whenever $|\gamma|+j \geq N$. Hence, by \eqref{separation of sum},
		\begin{eqnarray}\label{class for third term}
			\sum_{|\gamma|<N} \frac{(-i)^{|\gamma|}}{\gamma !} \sum_{j=0}^{N-1}(\Delta_{k}^{\gamma} \tau_{j})\left(\partial_{x}^{\gamma} \sigma\right)-1 \in M_{\rho,\Lambda}^{-\rho N}(\mathbb{T} \times \mathbb{Z}),
		\end{eqnarray}
		for every positive integer $N$. Thus, by \eqref{lambda minus forward difference tau class} and \eqref{class for third term},
		$$
		\lambda-1 \in M_{\rho,\Lambda}^{-\rho N}(\mathbb{T} \times \mathbb{Z}),
		$$
		for every positive integer $N$. Hence, if we pick $R$ to be the pseudo-differential operator with symbol $\lambda-1$, then the proof of \eqref{left side parametrix} is complete.
		
		By a similar argument, we can find another symbol, $\kappa \in M_{\rho,\Lambda}^{-m}(\mathbb{T} \times \mathbb{Z})$, such that
		\begin{eqnarray}\label{second left side parametrix}
			T_{\sigma} T_{\kappa}=I+R^{\prime},
		\end{eqnarray}
		where $R^{\prime}$ is a pseudo-differential operator with symbol in $\bigcap_{m \in \mathbb{R}} M_{\rho,\Lambda}^{m}(\mathbb{T} \times \mathbb{Z})$. By \eqref{left side parametrix} and \eqref{second left side parametrix},
		$$
		T_{\kappa}+R T_{\kappa}=T_{\tau}+T_{\tau} R^{\prime} .
		$$
		Since $R T_{\kappa}$ and $T_{\tau} R^{\prime}$ are pseudo-differential operators with symbols in $\cap_{m \in \mathbb{R}} M_{\rho,\Lambda}^{m}(\mathbb{T} \times \mathbb{Z})$, it follows that
		\begin{eqnarray}\label{relation in kappa and tau}
			T_{\kappa}=T_{\tau}+R^{\prime \prime},
		\end{eqnarray}
		where
		$$
		R^{\prime \prime}=T_{\tau} R^{\prime}-R T_{\kappa},
		$$
		is another pseudo-differential operator with symbol in $\bigcap_{m \in \mathbb{R}} M_{\rho,\Lambda}^{m}(\mathbb{T} \times \mathbb{Z})$. Hence, by \eqref{second left side parametrix} and \eqref{relation in kappa and tau},
		$$
		T_{\sigma} T_{\tau}=I+S,
		$$
		where
		$$
		S=R^{\prime}-T_{\sigma} R^{\prime \prime} .
		$$
		Since $S$ is a pseudo-differential operator with symbol in $\bigcap_{m \in \mathbb{R}} M_{\rho,\Lambda}^{m}(\mathbb{T} \times \mathbb{Z})$, it follows that \eqref{right side parametrix} is proved.
	\end{pff}
	\section{Compact   $M$-elliptic pseudo-differential operators}\label{cpt SG}\label{sec4}
	This section  is devoted to  study the compact   $M$-elliptic pseudo-differential operators on $\mathbb{ T}$ and $\mathbb{ Z}$.  First, we prove  Gohberg's lemma  for pseudo-differential operators on $\mathbb{ T}$ and $\mathbb{ Z}$ with  symbol in the weighted symbol class  $M_{\rho, \Lambda}^0(\mathbb{ T}\times \mathbb{Z})$ and $M_{\rho, \Lambda}^0(\mathbb{ Z}\times \mathbb{T})$, respectively. Using Gohberg's lemma,   we provide a necessary and sufficient condition to  ensure  the compactness  of  a  pseudo-differential operator  on   $L^2(\mathbb{T})$  and  $\ell^2(\mathbb{Z})$. We start this section by recalling the definition of Fredholm operators.

	Let $X$ and $Y$ be two complex Banach spaces. A closed linear operator $A : X \rightarrow Y$ with dense domain $\mathcal{D}(A)$ is said to be Fredholm if the null space $N(A)$ of $A$, and the null space $N\left(A^{t}\right)$ of the adjoint $A^{t}$ of $A$ are finite-dimensional, and the range space $R(A)$ of $A$ is a closed subspace of $Y.$ For a Fredholm operator $A$, the index $i(A)$ of $A$ is defined by
	$$
	i(A)=\operatorname{dim} N(A)-\operatorname{dim} N(A^{t}).
	$$
	Let $X$ be a complex Banach space and $A$: $X \rightarrow X$ be a closed linear operator with dense domain $\mathcal{D}(A)$. Then the spectrum $\Sigma(A)$ of $A$ is defined by
	$$
	\Sigma(A)=\mathbb{C} \backslash \rho(A),
	$$
	where $\rho(A)$ is the resolvent set of $A$ given by
	$$
	\rho(A)=\{\lambda \in \mathbb{C}: A-\lambda I \text { is bijective}\},
	$$
	and $I$ is the identity operator on $X$. The essential spectrum $\Sigma_{w}(A)$ of $A$ (defined in \cite{Wolf}) is given by
	$$
	\Sigma_{w}(A)=\mathbb{C} \backslash \Phi_{w}(A),
	$$
	where
	$$
	\Phi_{w}(A)=\{\lambda \in \mathbb{C}: A-\lambda I \text { is Fredholm}\}.
	$$
	For a detailed study on essential spectrum of an operator, we refer to \cite{Schechter,Schechterbook}.\\
	
	For $-\infty<s<\infty$, let $J_{s}$ be the pseudo-differential operator with symbol $\sigma_{s}$ given by
	$$
	\sigma_{s}(k)=\left(\Lambda(k)\right)^{-s}, \quad k \in \mathbb{Z}.
	$$
	Clearly, $\sigma_{s} \in M_{\rho,\Lambda}^{-s}\left(\mathbb{T} \times \mathbb{Z}\right)$. In literature, $J_{s}$ is called the Bessel potential of order $s$. For $s \in \mathbb{R}$ and $1\,<\,p\,<\,\infty$, we define the Sobolev space, $H_{\Lambda}^{s,p}$ by $$H_{\Lambda}^{s,p} = \{u \in \mathcal{D}^{\prime}(\mathbb{T}) : J_{-s}\,u \in \,L^p(\mathbb{T})\}.$$ 
	Then $H_{\Lambda}^{s,p}$ is a Banach space in which the norm $\|\,\|_{s,p}$ is given by
	$$\|u\|_{s,p,\Lambda} = \|J_{-s}\,u\|_{L^p(\mathbb{T})}, \quad u\,\in\,H_{\Lambda}^{s,p}.$$
	Note that $H_{\Lambda}^{0,p} = L^p(\mathbb{T}).$\\
	The following well-known compact Sobolev embedding theorem is just the weighted version of \cite[Theorem 2.5]{SM&MW}.
	\begin{theorem}\label{cpt thm}
		Let $s<t$. Then the inclusion $i: H_{\Lambda}^{t, p} \hookrightarrow H_{\Lambda}^{s, p}$ is compact for $1 < p < \infty$.
	\end{theorem}
The following boundedness result on weighted Sobolev space can be found in \cite{kal}.
	\begin{theorem}\label{bdd thm}
		Let $\sigma \in M_{\rho,\Lambda}^{m}\left(\mathbb{T} \times \mathbb{Z}\right),-\infty<m<\infty$. Then $T_{\sigma}: H_{\Lambda}^{s, p} \rightarrow H_{\Lambda}^{s-m, p}$ is a bounded linear operator for $1<p<\infty$.
	\end{theorem}
	The following proposition is the weighted version of \cite[Theorem 4.5]{MP} for $M_{\rho,\Lambda}^{0}\left(\mathbb{T} \times \mathbb{Z}\right)$ class.
	\begin{prop}\label{essential spectrum}
		Let $\sigma \in M_{\rho,\Lambda}^{0}\left(\mathbb{T} \times \mathbb{Z}\right)$ be such that
		$$
		\lim _{|k| \rightarrow \infty}\left\{\sup _{x \in[-\pi, \pi]}|\sigma(x, k)|\right\}=0 .
		$$
		Then
		$$
		\Sigma_{e}\left(T_{\sigma}\right)=\{0\} .
		$$
	\end{prop}
	A bounded linear operator $A$ on a complex separable and infinite-dimensional Hilbert space $X$ is essentially normal if $A A^{t}-A^{t} A$ is compact.
	
	The next result is about the essential normality of pseudo-differential operators with wrighted symbol of order 0.
	\begin{prop}\label{essential normal}
		Let $\sigma \in M_{\rho,\Lambda}^{0}\left(\mathbb{T} \times \mathbb{Z}\right)$. Then the bounded linear operator $T_{\sigma}$ : $L^{2}\left(\mathbb{T}\right) \rightarrow L^{2}\left(\mathbb{T}\right)$ is essentially normal.
	\end{prop}
	\begin{pff}
		Let $\tau \in M_{\rho,\Lambda}^{0}\left(\mathbb{T} \times \mathbb{Z}\right)$ be such that $T_{\sigma}^{t}=T_{\tau}$. Then using Theorem \ref{product}, we have
		$$
		T_{\sigma} T_{\tau}=T_{\gamma} \quad \text { and } \quad T_{\tau} T_{\sigma}=T_{\tilde{\gamma}} \text {, }
		$$
		where $\gamma$ and $\tilde{\gamma}$ are symbols of order 0. Moreover, $\gamma-\sigma \tau \in M_{\rho,\Lambda}^{-\rho}\left(\mathbb{T} \times \mathbb{Z}\right)$ and $\tilde{\gamma}-\sigma \tau \in M_{\rho,\Lambda}^{-\rho}\left(\mathbb{T} \times \mathbb{Z}\right)$. Therefore, $\gamma-\tilde{\gamma} \in M_{\rho,\Lambda}^{-\rho}\left(\mathbb{T} \times \mathbb{Z}\right)$. Hence, by Theorem \ref{cpt thm} and \ref{bdd thm}, we get
		$$
		T_{\sigma} T_{\sigma}^{t}-T_{\sigma}^{t} T_{\sigma}=T_{\gamma-\tilde{\gamma}}: L^{2}\left(\mathbb{T}\right) \rightarrow H_{\Lambda}^{\rho,2} \hookrightarrow L^{2}\left(\mathbb{T}\right)
		$$
		is compact, which completes the proof.
	\end{pff}
	
	The following theorem is known as Gohberg's lemma in the literature.
	\begin{theorem}\label{gohberg lemma}
		Let $\sigma \in M_{\rho,\Lambda}^{0}\left(\mathbb{T} \times \mathbb{Z}\right)$. Then for all compact operators $K$ on $L^{2}\left(\mathbb{T}\right)$,
		\begin{eqnarray}\label{gohberg inequality}
			\left\|T_{\sigma}-K\right\|_{*} \geq d,
		\end{eqnarray}
		where
		$$
		d=\limsup _{|k| \rightarrow \infty}\left\{\sup _{x \in[-\pi, \pi]}|\sigma(x,k)|\right\} .
		$$
		Here $\|\cdot\|_{*}$ denotes the norm in the $C^{*}$-algebra of all bounded linear operators on $L^2(\mathbb{T})$.
	\end{theorem}
	\begin{pff}
		Let $u$ be a nonzero function in $C^{\infty}\left(\mathbb{T}\right)$. Then
		$$
		\begin{aligned}
			& \left(T_{\sigma} u\right)(x)=(2 \pi)^{-1} \sum_{k \in \mathbb{Z}}\left\{\int_{-\pi}^{\pi} e^{i k(x-y)} \sigma(x, k) u(y) d y\right\} \\
			& =(2 \pi)^{-1} \int_{-\pi}^{\pi}\left(\mathcal{F}_{\mathbb{Z}} \sigma)(x, x-y) u(y) d y\right. \\
			& =(2 \pi)^{-1} \int_{-\pi}^{\pi}\left(\mathcal{F}_{\mathbb{Z}} \sigma\right)(x, y) u(x-y) d y, \quad x \in[-\pi, \pi], 
		\end{aligned}
		$$
		where $\mathcal{F}_{\mathbb{Z}} \sigma$ is the Fourier transform of $\sigma$ with respect to the second variable in the sense of distribution. Since $\sigma \in C^{\infty}\left(\mathbb{T} \times \mathbb{Z}\right)$, it follows that for all $k \in \mathbb{Z}$, there exists $x_{k} \in[-\pi, \pi]$ such that
		$$
		\left|\sigma\left(x_{k}, k\right)\right|=\sup _{x \in[-\pi, \pi]}|\sigma(x, k)| .
		$$
		By the definition of $d$, there exists a sequence $\left\{\left(x_{k_{m}}, k_{m}\right)\right\}_{m=1}^{\infty}$ such that
		$$
		\left|k_{m}\right| \rightarrow \infty,
		$$
		and
		$$
		\left|\sigma\left(x_{k_{m}}, k_{m}\right)\right| \rightarrow d
		$$
		as $m \rightarrow \infty$. For $m=1,2, \ldots$, we define the function $u_{k_{m}}$ on $\mathbb{T}$ by
		$$
		u_{k_{m}}(x)=e^{i k_{m} x} u\left(x-x_{k_{m}}\right), \quad x \in[-\pi, \pi] .
		$$
		Then
		$$
		\left\|u_{k_{m}}\right\|_{L^{2}\left(\mathbb{T}\right)}=\|u\|_{L^{2}\left(\mathbb{T}\right)}, \quad m=1,2, \ldots,
		$$
		and moreover, using the Riemann-Lebesgue lemma, one can easily show that $u_{k_{m}} \rightarrow 0$ weakly as $k \rightarrow \infty$. Let $K : L^{2}(\mathbb{ T}) \rightarrow L^{2}(\mathbb{ T})$ be a compact operator and $\epsilon$ be an arbitrary positive number. Then
		$$
		\left\|K u_{k_{m}}\right\|_{L^{2}\left(\mathbb{T}\right)} \rightarrow 0
		$$
		as $m \rightarrow \infty$, and hence, for sufficiently large $m$,
		\begin{eqnarray}\label{cpt operator estimate}
			\left\|K u_{k_{m}}\right\|_{L^{2}\left(\mathbb{T}\right)} \leq \epsilon\|u\|_{L^{2}\left(\mathbb{T}\right)} .
		\end{eqnarray}
		\begin{lemma}\label{L^2 convergence of sigma}
			$\left\|\sigma\left(\cdot, k_{m}\right) u_{k_{m}}-T_{\sigma} u_{k_{m}}\right\|_{L^{2}\left(\mathbb{T}\right)} \rightarrow 0$ as $m \rightarrow \infty$.
		\end{lemma}
		We assume the lemma for a moment and continue with the proof of Theorem \ref{gohberg lemma}. By Lemma \ref{L^2 convergence of sigma}, we have
		\begin{eqnarray}\label{lemma estimate}
			\left\|\sigma\left(\cdot, k_{m}\right) u_{k_{m}}\right\|_{L^{2}\left(\mathbb{T}\right)}-\left\|T_{\sigma} u_{k_{m}}\right\|_{L^{2}\left(\mathbb{T}\right)} \leq \epsilon\|u\|_{L^{2}\left(\mathbb{T}\right)},
		\end{eqnarray}
		for sufficiently large $m$. Since $\sigma \in C^{\infty}\left(\mathbb{T} \times \mathbb{Z}\right)$, so all derivatives of $\sigma\left(\cdot, k_{m}\right)$ exist and are bounded, and hence, there exists a positive number $\delta$ such that for all $x$ in $\left[-\pi+x_{k_{m}}, \pi+x_{k_{m}}\right]$ with $\left|x-x_{k_{m}}\right|<\delta$, we have
		\begin{eqnarray}\label{uniform continuity of sigma}
			\left|\sigma\left(x, k_{m}\right)-\sigma\left(x_{k_{m}}, k_{m}\right)\right|<\epsilon .
		\end{eqnarray}
		Choose $u \in C^{\infty}\left(\mathbb{T}\right)$ be such that
		$$
		u(x)=0, \quad|x| \geq \delta .
		$$
		Then $u_{k_{m}}(x)=0$ for all $x$ in $\left[-\pi+x_{k_{m}}, \pi+x_{k_{m}}\right]$ with $\left|x-x_{k_{m}}\right| \geq \delta .$ So,
		$$
		\begin{aligned}
			& \left\|\sigma\left(x_{k_{m}}, k_{m}\right) u_{k_{m}}\right\|_{L^{2}\left(\mathbb{T}\right)}-\left\|\sigma\left(\cdot, k_{m}\right) u_{k_{m}}\right\|_{L^{2}\left(\mathbb{T}\right)} \\
			& \leq\left\|\sigma\left(x_{k_{m}}, k_{m}\right) u_{k_{m}}-\sigma\left(\cdot, k_{m}\right) u_{k_{m}}\right\|_{L^{2}\left(\mathbb{T}\right)} \\
			& =\left\{\int_{-\pi+x_{k_{m}}}^{\pi+x_{k_{m}}}\left|\sigma\left(x, k_{m}\right)-\sigma\left(x_{k_{m}}, k_{m}\right)\right|^{2}\left|u_{k_{m}}(x)\right|^{2} d x\right\}^{1/2} \\
			& =\left\{\int_{\left\{x \in\left[-\pi+x_{k_{m}}, \pi+x_{k_{m}}\right]:\left|x-x_{k_{m}}\right|<\delta\right\}}\left|\sigma\left(x, k_{m}\right)-\sigma\left(x_{k_{m}}, k_{m}\right)\right|^{2}\left|u_{k_{m}}(x)\right|^{2} d x\right\}^{1 / 2} . 
		\end{aligned}
		$$
		Hence, by \eqref{uniform continuity of sigma},
		\begin{eqnarray}\label{difference of symbol estimate}
			\left|\sigma\left(x_{k_{m}}, k_{m}\right)\right|\|u\|_{L^{2}\left(\mathbb{T}\right)}-\left\|\sigma\left(\cdot, k_{m}\right) u_{k_{m}}\right\|_{L^{2}\left(\mathbb{T}\right)} \leq \epsilon\|u\|_{L^{2}\left(\mathbb{T}\right)}.
		\end{eqnarray}
		Thus, by \eqref{cpt operator estimate}, \eqref{lemma estimate}, and \eqref{difference of symbol estimate}, for sufficiently large $m$, we get
		$$
		\begin{aligned}
			\|u\|_{L^{2}\left(\mathbb{T}\right)}\left\|T_{\sigma}-K\right\|_{*} & \geq\left\|\left(T_{\sigma}-K\right) u_{k_{m}}\right\|_{L^{2}\left(\mathbb{T}\right)} \\
			& \geq\left\|T_{\sigma} u_{k_{m}}\right\|_{L^{2}\left(\mathbb{T}\right)}-\left\|K u_{k_{m}}\right\|_{L^{2}\left(\mathbb{T}\right)} \\
			& \geq\left\|T_{\sigma} u_{k_{m}}\right\|_{L^{2}\left(\mathbb{T}\right)}-\epsilon\|u\|_{L^{2}\left(\mathbb{T}\right)} \\
			& \geq\left\|\sigma\left(\cdot, k_{m}\right) u_{k_{m}}\right\|_{L^{2}\left(\mathbb{T}\right)}-2 \epsilon\|u\|_{L^{2}\left(\mathbb{T}\right)} \\
			& \geq\left|\sigma\left(x_{k_{m}}, k_{m}\right)\right|\|u\|_{L^{2}\left(\mathbb{T}\right)}-3 \epsilon\|u\|_{L^{2}\left(\mathbb{T}\right)} \\
			&=\left(\left|\sigma\left(x_{k_{m}}, k_{m}\right)\right|-3 \epsilon\right)\|u\|_{L^{2}\left(\mathbb{T}\right)}.
		\end{aligned}
		$$
		Letting $m \rightarrow \infty$, we get
		$$
		\left\|T_{\sigma}-K\right\|_{*} \geq d-3 \epsilon.
		$$
		Finally, using the fact that $\epsilon$ is an arbitrary positive number, we have
		$$
		\left\|T_{\sigma}-K\right\|_{*} \geq d.
		$$
	\end{pff}
	\begin{proof}[Proof of Lemma \ref{L^2 convergence of sigma}] Let
		$$
		K(x, y)=\left(\mathcal{F}_{\mathbb{Z}} \sigma\right)(x, y), \quad x, y \in[-\pi, \pi].
		$$
		Then for $k=1,2, \ldots$,
		\begin{eqnarray}\label{pdo on defined sequence}
			\left(T_{\sigma} u_{k_{m}}\right)(x)=(2 \pi)^{-1} \int_{-\pi}^{\pi} e^{i k_{m}(x-y)} K(x, y) u_{x_{k_{m}}}(x-y) d y, \quad x \in[-\pi, \pi],
		\end{eqnarray}
		where $u_{x_{k_m}} = u (x-x_{k_m}).$ In the view of Example 2.4 in \cite{MR&VT&JW}, the function $q(x) = e^{-2 \pi i x} -1$ gives rise to a strongly admissible difference operator on $\mathbb{T}$ with the property that
		$$
		\Delta p(k) = \Delta_{q} p(k) = p(k+1) - p(k), \quad k \in \mathbb{Z}.
		$$ 
		Let $N$ be a positive integer. Then by the Taylor expansion formula (see \cite{MR&VT&JW})
		\begin{eqnarray}\label{taylor formula}
			u_{x_{k_{m}}}(x-y)=u_{x_{k_{m}}}(x)+\sum_{\alpha=1}^{N-1} \frac{1}{\alpha !} q^{\alpha}(y) \partial^{\alpha} u_{x_{k_{m}}}(x)+\mathcal{O}\left(h(y)^{N}\right),
		\end{eqnarray}
		where $h(x)$ is the geodesic distance from $x$ and the identity element of $\mathbb{T}$. Using \eqref{pdo on defined sequence} and \eqref{taylor formula}, we can write
		$$
		\begin{aligned}
			T_{\sigma} u_{k_{m}}(x)&=(2 \pi)^{-1}\int_{\mathbb{T}} e^{i k_{m}(x-y)} K(x,y) u_{x_{k_{m}}}(x) d y\\ &+(2 \pi)^{-1}\int_{\mathbb{T}} e^{i k_{m}(x-y)} K(x,y) \sum_{\alpha=1}^{N-1} \frac{1}{\alpha !} q^{\alpha}(y) \partial^{\alpha} u_{x_{k_{m}}}(x) d y \\
			&+(2 \pi)^{-1}\int_{\mathbb{T}} e^{i k_{m}(x-y)} K(x,y) \mathcal{O}\left(h(y)^{N}\right) d y .
		\end{aligned}
		$$
		Define
		$$
		\begin{aligned}
			&I_{1}:=\int_{\mathbb{T}} e^{i k_{m}(x-y)} K(x,y) u_{x_{k_{m}}}(x) d y, \\
			&I_{2}:=\int_{\mathbb{T}} e^{i k_{m}(x-y)} K(x,y) \sum_{\alpha=1}^{N-1} \frac{1}{\alpha !} q^{\alpha}(y) \partial^{\alpha} u_{x_{k_{m}}}(x) d y,
		\end{aligned}
		$$
		and
		$$
		I_{3}:=\int_{\mathbb{T}} e^{i k_{m}(x-y)} K(x,y) \mathcal{O}\left(h(y)^{N}\right) d y.
		$$
		So, we have
		$$
		\begin{aligned}
			I_{1} &=\int_{\mathbb{T}} e^{i k_{m}(x-y)} K(x,y) u_{x_{k_{m}}}(x) d y \\
			&=u_{k_{m}}(x)\sigma\left(x,k_{m} \right), \\
			I_{2} &=\int_{\mathbb{T}} e^{i k_{m}(x-y)} K(x,y) \sum_{\alpha=1}^{N-1} \frac{1}{\alpha !} q^{\alpha}(y)  \partial^{\alpha} u_{x_{k_{m}}}(x) d y \\
			&=\sum_{\alpha=1}^{N-1} \frac{1}{\alpha !} e^{i k_{m}x}  \partial^{\alpha} u_{x_{k_{m}}}(x) \int_{\mathbb{T}} e^{-i k_{m}y} K(x,y) q^{\alpha}(y) d y  \\
			&=\sum_{\alpha=1}^{N-1} \frac{1}{\alpha !} e^{i k_{m}x} \partial^{\alpha} u_{x_{k_{m}}}(x) \Delta_{q}^{\alpha} \sigma\left(x, k_{m}\right)\\
			&=\sum_{\alpha=1}^{N-1} \frac{1}{\alpha !} e^{i k_{m}x}  \partial^{\alpha} u_{x_{k_{m}}}(x) \Delta^{\alpha} \sigma\left(x, k_{m}\right),\\
			I_{3} &=\int_{\mathbb{T}} K(x,y) \mathcal{O}\left(h(y)^{N}\right) e^{i k_{m}(x-y)} d y\\
			&=\int_{\mathbb{T}} K(x,y) q^{N}(y) e^{i k_{m}(x-y)} d y \\
			&=e^{i k_{m}x}\Delta_{q}^{N} \sigma\left(x, k_{m}\right)\\
			&=e^{i k_{m}x}\Delta^{N} \sigma\left(x, k_{m}\right),
		\end{aligned}
		$$
		where $q^{N}(x):=\mathcal{O}\left(h(x)^{N}\right)$ vanishes at the identity element of $\mathbb{T}$. Hence,
		$$
		T_{\sigma} u_{k_{m}}(x)- u_{k_{m}}(x)\sigma\left(x,k_{m} \right) = \sum_{\alpha=1}^{N-1} \frac{1}{\alpha !} e^{i k_{m}x} \ \partial^{\alpha} u_{x_{k_{m}}}(x) \Delta^{\alpha} \sigma\left(x, k_{m}\right)+e^{i k_{m}x} \Delta^{N} \sigma\left(x, k_{m}\right) .
		$$
		Define $$T_{N}^{1}(x):=\sum_{\alpha=1}^{N-1} \frac{1}{\alpha !} e^{i k_{m}x} \partial^{\alpha} u_{x_{k_{m}}}(x) \Delta^{\alpha} \sigma\left(x, k_{m}\right),$$
		 and 
		 $$T_{N}^{2}(x):=e^{i k_{m}x} \Delta^{N} \sigma\left(x, k_{m}\right).$$ Now using the fact that $u_{x_{k_{m}}} \in C^{\infty}(\mathbb{T})$, we obtain
		$$
		\begin{aligned}
			|T_{N}^{1}(x)| & \leq \sum_{\alpha=1}^{N-1} C_{\alpha} |\Delta^{\alpha} \sigma\left(x, k_{m}\right)|\\ 
			& \leq \sum_{\alpha=1}^{N-1} C_{\alpha} \,\Lambda(k_m)^{-\rho \alpha}\\
			& \leq C \,\Lambda(k_m)^{-\rho}, \quad  x \in \mathbb{T},
		\end{aligned}
		$$
		where $C = \sum_{\alpha=1}^{N-1} C_{\alpha}.$ Clearly $|T_{N}^{1}(x)| \rightarrow 0$ uniformly on $\mathbb{T}$ as $m \rightarrow \infty$, and hence, $\|T_{N}^{1}\|_{L^{2}(\mathbb{T})} \rightarrow 0$ as $m \rightarrow \infty$.\\
		Similarly, $|T_{N}^{2}(x)|  \leq C_{0} \,\Lambda(k_m)^{-\rho\,N} \leq C_{0} \,\Lambda(k_m)^{-\rho}\rightarrow 0$ uniformly on $\mathbb{T}$ as $m \rightarrow \infty$, and hence, $\|T_{N}^{2}\|_{L^{2}(\mathbb{T})} \rightarrow 0$ as $m \rightarrow \infty$. This completes the proof of the theorem.
	\end{proof}
	Now we recall the definition of the Calkin algebra, which will be used to prove the main result of this section. Let $B\left(L^{2}\left(\mathbb{T}\right)\right)$ and $K\left(L^{2}\left(\mathbb{T}\right)\right)$ denotes the $C^{*}$-algebra of bounded linear operators on $L^{2}\left(\mathbb{T}\right)$ and the ideal of compact operators on $L^{2}\left(\mathbb{T}\right)$, respectively. The Calkin algebra $B\left(L^{2}\left(\mathbb{T}\right)\right) / K\left(L^{2}\left(\mathbb{T}\right)\right)$ is a *-algebra with respect to the product and the adjoint defined as follows:
	$$
	[A][B]=[A B]
	$$
	and
	$$
	[A]^{*}=\left[A^{*}\right]
	$$
	for all $A$ and $B$ in $B\left(L^{2}\left(\mathbb{T}\right)\right).$ Two elements $[A]$ and $[B]$ be in the Calkin algebra $B\left(L^{2}\left(\mathbb{T}\right)\right) / K\left(L^{2}\left(\mathbb{T}\right)\right)$ are equal if and only if $ A-B \in K\left(L^{2}\left(\mathbb{T}\right)\right)$.
	
	The norm $\|\cdot\|_{C}$ in $B\left(L^{2}\left(\mathbb{T}\right)\right) / K\left(L^{2}\left(\mathbb{T}\right)\right)$ is given by
	$$
	\|[A]\|_{C}=\inf _{K \in K\left(L^{2}\left(\mathbb{T}\right)\right)}\|A-K\|_{*}, \quad[A] \in B\left(L^{2}\left(\mathbb{T}\right)\right) / K\left(L^{2}\left(\mathbb{T}\right)\right) .
	$$
	It can be shown that $B\left(L^{2}\left(\mathbb{T}\right)\right) / K\left(L^{2}\left(\mathbb{T}\right)\right)$ is a $C^{*}$-algebra. Using the Calkin algebra, \eqref{gohberg inequality} in Gohberg's lemma is the same as
	$$
	\left\|\left[T_{\sigma}\right]\right\|_{C} \geq d .
	$$
	Now we are ready to prove our main theorem in this section.
	\begin{theorem}\label{characterization of cpt operator on T}
		Let $\sigma \in M_{\rho,\Lambda}^{0}\left(\mathbb{T} \times \mathbb{Z}\right)$. Then $T_{\sigma}$ is a compact operator on $L^{2}\left(\mathbb{T}\right)$ if and only if $d=0$, where
		$$
		d=\limsup _{|k| \rightarrow \infty}\left\{\sup _{x \in[-\pi, \pi]}|\sigma(x,k)|\right\} .
		$$
	\end{theorem}
	\begin{pff}
		First, let us assume that $d=0$. Then $T_{\sigma}$ is compact if and only if $\left[T_{\sigma}\right]=0$ in $B\left(L^{2}\left(\mathbb{T}\right)\right) / K\left(L^{2}\left(\mathbb{T}\right)\right)$. By Proposition \ref{essential normal}, $T_{\sigma}$ is essentially normal on $L^{2}\left(\mathbb{T}\right)$, which implies that $\left[T_{\sigma}\right]$ is normal in the Calkin algebra $B\left(L^{2}\left(\mathbb{T}\right)\right) / K\left(L^{2}\left(\mathbb{T}\right)\right)$. Hence,
		$$
		r\left(\left[T_{\sigma}\right]\right)=\left\|\left[T_{\sigma}\right]\right\|_{C},
		$$
		where $r\left(\left[T_{\sigma}\right]\right)$ is the spectral radius of $\left[T_{\sigma}\right]$, and by Proposition \ref{essential spectrum}, we get $\Sigma_{e}\left(T_{\sigma}\right)=\{0\}$. Therefore, by Atkinson's theorem (see \cite{Atkinson}), the spectrum of $\left[T_{\sigma}\right]$ in the Calkin algebra $B\left(L^{2}\left(\mathbb{T}\right)\right) / K\left(L^{2}\left(\mathbb{T}\right)\right)$ is
		$$
		\Sigma\left(\left[T_{\sigma}\right]\right)=\{0\} .
		$$
		This implies that
		$$
		\left\|\left[T_{\sigma}\right]\right\|_{C}=r\left(\left[T_{\sigma}\right]\right)=0 .
		$$
		Hence, it follows that
		$$
		\left[T_{\sigma}\right]=0 .
		$$
		Therefore $T_{\sigma}$ is compact.\\
		Now, to prove the converse part, assume that $d \neq 0$, then we need to show that $T_{\sigma}$ is not compact on $L^{2}\left(\mathbb{T}\right)$. Suppose that $T_{\sigma}$ is compact, then putting $K=T_{\sigma}$ in \eqref{gohberg inequality} will contradict our assumption that $d \neq 0$. This completes the proof of the theorem.
	\end{pff}
	
	Now our aim is to study the Gohberg's lemma and characterization of compact operators on $\ell^{2}(\mathbb{ Z})$ with symbol in $M_{\rho,\Lambda}^{0}\left(\mathbb{Z} \times \mathbb{T}\right)$ class (defined in \cite{SSM&VK}). The main ingredient is the relation between the weighted periodic and discrete symbols which can be found in \cite{SSM&VK}.
	\begin{theorem}\label{relation between tau and sigma}
		Let $\sigma: \mathbb{Z} \times \mathbb{T} \rightarrow \mathbb{C}$ be a measurable function such that the pseudo-differential operator $T_{\sigma}: \ell^{2}\left(\mathbb{Z}\right) \rightarrow \ell^{2}\left(\mathbb{Z}\right)$ is a bounded linear operator. If we define $\tau: \mathbb{T}\times \mathbb{Z} \rightarrow \mathbb{C}$ by
		$$
		\tau(x, k)=\overline{\sigma(-k, x)}, \quad  x \in \mathbb{T}, k \in \mathbb{Z},	
		$$
		then 
		\begin{eqnarray}\label{sigma in term of tau}
			T_{\sigma}=\mathcal{F}_{\mathbb{Z}}^{-1} T_{\tau}^{*} \mathcal{F}_{\mathbb{Z}},
		\end{eqnarray}
		where $T_{\tau}^{*}$ is the adjoint of $T_{\tau}$. We also have
		\begin{eqnarray}\label{tau in term of sigma}
			T_{\tau}=\mathcal{F}_{\mathbb{Z}} T_{\sigma}^{*} \mathcal{F}_{\mathbb{Z}}^{-1},
		\end{eqnarray}
		where $T_{\sigma}^{*}$ is the adjoint of $T_{\sigma}$.
	\end{theorem}
As a corollary of Theorem \ref{gohberg lemma} and Theorem \ref{relation between tau and sigma}, we obtain the following estimates for the distance between a given operator and
the space of compact operators on $ \ell^2(\mathbb{Z}).$
	\begin{corollary}\label{gohberg lemma on Z}
		Let $\sigma \in M_{\rho,\Lambda}^{0}\left(\mathbb{Z} \times \mathbb{T}\right)$. Then for all compact operators $K$ on $\ell^{2}\left(\mathbb{Z}\right)$,
		\begin{eqnarray}\label{gohberg inequality on Z}
			\left\|T_{\sigma}-K\right\|_{**} \geq d,
		\end{eqnarray}
		where
		$$
		d=\limsup _{|k| \rightarrow \infty}\left\{\sup _{x \in[-\pi, \pi]}|\sigma(k,x)|\right\} .
		$$
		Here $\|\cdot\|_{**}$ denotes the norm in the $C^{*}$-algebra of all bounded linear operators on $\ell^2(\mathbb{Z})$.
	\end{corollary}
	\begin{pff}
		We define $\tau: \mathbb{T}\times \mathbb{Z} \rightarrow \mathbb{C}$ by
		$$	\tau(x, k)=\overline{\sigma(-k, x)}, \quad  x \in \mathbb{T}, k \in \mathbb{Z}.$$
		Since $\sigma \in M_{\rho,\Lambda}^{0}\left(\mathbb{Z} \times \mathbb{T}\right)$, we have $\tau \in M_{\rho,\Lambda}^{0}\left(\mathbb{T} \times \mathbb{Z}\right)$. Then, by Theorem \ref{gohberg lemma}, we have the following estimate for all compact operators $K^{\prime}$ on $L^{2}\left(\mathbb{T}\right):$
		\begin{eqnarray}\label{gohberg inequality in the T case which used in Z case}
			\left\|T_{\tau}-K^{\prime}\right\|_{*} \geq d.
		\end{eqnarray}
		We need to show that 
		\begin{eqnarray}\label{gohberg inequality need to prove in Z case}
			\left\|T_{\sigma}-K\right\|_{**} \geq d,
		\end{eqnarray}
		for all compact operators $K$ on $\ell^{2}\left(\mathbb{Z}\right).$ Let $K : \ell^{2}\left(\mathbb{Z}\right) \rightarrow \ell^{2}\left(\mathbb{Z}\right)$ be any arbitrary compact operator. This implies that $K^{*} : \ell^{2}\left(\mathbb{Z}\right) \rightarrow \ell^{2}\left(\mathbb{Z}\right)$ is also a compact operator. Then $K_{1} = \mathcal{F}_{\mathbb{Z}} K^{*} \mathcal{F}_{\mathbb{Z}}^{-1} : L^2(\mathbb{T}) \rightarrow L^2(\mathbb{T})$ is a compact operator. Hence, by \eqref{tau in term of sigma} and \eqref{gohberg inequality in the T case which used in Z case}, we have
		$$
		\begin{aligned}
			& \hspace{28.1pt}\left\|T_{\tau}-K_{1}\right\|_{L^2(\mathbb{T})} \geq d,\\
			& \implies \left\| \mathcal{F}_{\mathbb{Z}} T_{\sigma}^{*} \mathcal{F}_{\mathbb{Z}}^{-1}- \mathcal{F}_{\mathbb{Z}} K^{*} \mathcal{F}_{\mathbb{Z}}^{-1}\right\|_{L^2(\mathbb{T})} \geq d,\\
			&\implies \left\|\mathcal{F}_{\mathbb{Z}} (T_{\sigma}- K)^{*} \mathcal{F}_{\mathbb{Z}}^{-1}\right\|_{L^2(\mathbb{T})} \geq d,\\
			&\implies \left\|(T_{\sigma}-K)^{*}\right\|_{\ell^2(\mathbb{Z})} \geq d,\\
			&\implies \left\|T_{\sigma}-K\right\|_{\ell^2(\mathbb{Z})} \geq d,
		\end{aligned}
		$$
		and this completes the proof of the estimate \eqref{gohberg inequality need to prove in Z case}.
	\end{pff}

The following corollary gives us a necessary and sufficient condition for an operator to be compact on $ \ell^2(\mathbb{Z})$ for symbol class $M_{\rho,\Lambda}^{0}\left(\mathbb{Z} \times \mathbb{T}\right).$
	\begin{corollary}\label{characterization of cpt operator on Z}
		Let $\sigma \in M_{\rho,\Lambda}^{0}\left(\mathbb{Z} \times \mathbb{T}\right)$. Then $T_{\sigma}$ is a compact operator on $\ell^{2}\left(\mathbb{Z}\right)$ if and only if $d=0$, where
		$$
		d=\limsup _{|k| \rightarrow \infty}\left\{\sup _{x \in[-\pi, \pi]}|\sigma(k,x)|\right\} .
		$$
	\end{corollary}
	\begin{pff}
		First, let us assume that $d = 0.$ Define
		$$	\tau(x, k)=\overline{\sigma(-k, x)}, \quad  x \in \mathbb{T}, k \in \mathbb{Z}.$$
		Since $\sigma \in M_{\rho,\Lambda}^{0}\left(\mathbb{Z} \times \mathbb{T}\right)$, we have $\tau \in M_{\rho,\Lambda}^{0}\left(\mathbb{T} \times \mathbb{Z}\right)$. Hence,  by Theorem \ref{characterization of cpt operator on T}, $T_{\sigma}$ is compact on $L^{2}\left(\mathbb{T}\right)$. This implies that $T_{\sigma}^{*}$ is a compact operator on $L^{2}\left(\mathbb{T}\right)$. Hence, by \eqref{sigma in term of tau}, $T_{\sigma}$ is a compact operator on $\ell^{2}\left(\mathbb{Z}\right)$.\\ Now, to prove the converse part, assume that $d \neq 0$, then we need to show that $T_{\sigma}$ is not compact on $\ell^{2}\left(\mathbb{Z}\right)$. Suppose that $T_{\sigma}$ is compact, then putting $K=T_{\sigma}$ in \eqref{gohberg inequality on Z} will contradict our assumption that $d \neq 0$. This completes the proof of the theorem.
	\end{pff}
	\section{G\r{a}rding's and sharp G\r{a}rding's inequalities on  $\mathbb{T}$  and $\mathbb{Z}$}\label{sec5}
	The main aim of this section is to prove G\r{a}rding's and sharp  G\r{a}rding's  inequalities for $M$-elliptic operators on $\mathbb{T}$ and  $\mathbb{Z}$, respectively.     First, we   state the G\r{a}rding's inequality for $M$-elliptic operators on $\mathbb{T}$, which is analogous to the \cite[Corollary 5.7]{AB&GK&MK}. The proof can be done in similar lines, so we skip the proof here.
	\begin{theorem} \textbf{(G\r{a}rding's inequality for $M$-elliptic operators on $\mathbb{T}$)}\label{Gardings theorem on T}\\ 
		Let $m>0$. Let $\sigma \in  M_{\rho, \Lambda}^{2 m}\left(\mathbb{T}\times \mathbb{Z}\right)$ be elliptic such that $\sigma(x, k) \geq 0$ for all $x$ and co-finitely many $k$. Then there exist $C_{0}, C_{1}>0$ such that for all $f \in H_{\Lambda}^{m,2}\left(\mathbb{T}\right)$, we have
		\begin{align}\label{garding}
			\operatorname{Re}(T_{\sigma} f, f)_{L^{2}\left(\mathbb{T}\right)} \geq C_{0}\|f\|_{H_{\Lambda}^{m,2}\left(\mathbb{T}\right)}^{2}-C_{1}\|f\|_{L^{2}\left(\mathbb{T}\right)}^{2}.
		\end{align}
	\end{theorem}
	
	Now, we will show that Theorem \ref{Gardings theorem on T} implies the corresponding Gårding inequality for $M$-elliptic operators on $\mathbb{Z}$. As there is no regularity concept on the lattice, the statement is given in terms of weighted $\ell^{2}$-spaces. For this, we need the following definition:
	\begin{defi}
		For $s \in \mathbb{R}$, let us define the weighted space $\ell_{s,\Lambda}^{2}\left(\mathbb{Z}\right)$ as the space of all $f: \mathbb{Z}\rightarrow \mathbb{C}$ such that
		$$
		\|f\|_{\ell_{s,\Lambda}^{2}\left(\mathbb{Z}\right)}:=\left(\sum_{k \in \mathbb{Z}} \Lambda(k)^{2 s}|f(k)|^{2}\right)^{1 / 2}<\infty .
		$$
		We observe that the symbol $\sigma_{s}(k)=\Lambda(k)^{s}$ belongs to $M_{\rho,\Lambda}^{s}\left(\mathbb{Z} \times \mathbb{T}\right)$, and $f \in \ell_{s,\Lambda}^{2}\left(\mathbb{Z}\right)$ if and only if $T_{\sigma_{s}} f \in \ell^{2}\left(\mathbb{Z}\right)$.
	\end{defi}
	\begin{theorem} \textbf{(G\r{a}rding's inequality for $M$-elliptic operators on $\mathbb{Z}$)}\label{Gardings theorem on Z}\\ 
		Let $m>0$. Let $\sigma \in  M_{\rho, \Lambda}^{2 m}\left(\mathbb{Z}\times \mathbb{T}\right)$ be elliptic such that $\sigma(k,x) \geq 0$ for all $x$ and co-finitely many $k$.  Then there exist $C_{1}, C_{2}>0$ such that for all $f \in \ell_{m,\Lambda}^{2}\left(\mathbb{Z}\right)$, we have
		$$
		\operatorname{Re}(T_{\sigma} f, f)_{\ell^{2}\left(\mathbb{Z}\right)} \geq C_{0}\|f\|_{\ell_{m,\Lambda}^{2}\left(\mathbb{Z}\right)}^{2}-C_{1}\|f\|_{\ell^{2}\left(\mathbb{Z}\right)}^{2} .
		$$
	\end{theorem}
	\begin{pff}
		Let $\tau(x, k)=\overline{\sigma(-k, x)}$. Then by Theorem  \ref{relation between tau and sigma}, we have
		$$
		T_{\sigma}=\mathcal{F}_{\mathbb{Z}}^{-1} T_{\tau}^{*} \mathcal{F}_{\mathbb{Z}},
		$$
		and if $\sigma$ is elliptic on $\mathbb{Z} \times \mathbb{T}$, then $\tau$ is elliptic on $\mathbb{T} \times \mathbb{Z}$. Also, if $\sigma \geq 0$, then $\tau \geq 0$. Then by Theorem \ref{Gardings theorem on T}, for all $g \in H_{\Lambda}^{m,2}\left(\mathbb{T}\right)$, we have
		\begin{eqnarray}\label{garding inequality form on g}
			\operatorname{Re}\left(T_{\tau}^{*} g, g\right)_{L^{2}\left(\mathbb{T}\right)}=\operatorname{Re}\left(T_{\tau} g, g\right)_{L^{2}\left(\mathbb{T}\right)} \geq C_{0}\|g\|_{H_{\Lambda}^{m,2}\left(\mathbb{T}\right)}^{2}-C_{1}\|g\|_{L^{2}\left(\mathbb{T}\right)}^{2}.
		\end{eqnarray}
		Let $f \in \ell_{m,\Lambda}^{2}\left(\mathbb{Z}\right)$ and $g=\mathcal{F}_{\mathbb{Z}} f$. Then $g \in H_{\Lambda}^{m,2}\left(\mathbb{T}\right)$, and
		\begin{eqnarray}\label{norm relation in garding inequality}
			\|g\|_{H_{\Lambda}^{m,2}\left(\mathbb{T}\right)}=\|f\|_{\ell_{m,\Lambda}^{2}\left(\mathbb{Z}\right)} \quad \text { and } \quad \|g\|_{L^{2}\left(\mathbb{T}\right)}=\|f\|_{\ell^{2}\left(\mathbb{Z}\right)} .
		\end{eqnarray}
		Now by Theorem \ref{relation between tau and sigma},
		$$
		T_{\sigma} f=\mathcal{F}_{\mathbb{Z}}^{-1} \circ T_{\tau}^{*} \circ \mathcal{F}_{\mathbb{Z}} f=\mathcal{F}_{\mathbb{Z}}^{-1} \circ T_{\tau}^{*} g,
		$$
		so that $\mathcal{F}_{\mathbb{Z}} T_{\sigma} f=T_{\tau}^{*} g .$ Substituting \eqref{norm relation in garding inequality} into \eqref{garding inequality form on g}, we get
		$$
		\begin{aligned}
			\operatorname{Re}\left(T_{\tau}^{*} g, g\right)_{L^{2}\left(\mathbb{T}\right)} & \geq C_{0}\|f\|_{\ell_{m,\Lambda}^{2}\left(\mathbb{Z}\right)}^{2}-C_{1}\|f\|_{\ell^{2}\left(\mathbb{Z}\right)}^{2}, \\
			\operatorname{Re}\left(\mathcal{F}_{\mathbb{Z}} T_{\sigma} f, \mathcal{F}_{\mathbb{Z}} f\right)_{L^{2}\left(\mathbb{T}\right)} & \geq C_{0}\|f\|_{\ell_{m,\Lambda}^{2}\left(\mathbb{Z}\right)}^{2}-C_{1}\|f\|_{\ell^{2}\left(\mathbb{Z}\right)}^{2}, \\
			\operatorname{Re}\left(\mathcal{F}_{\mathbb{Z}}^{*} \mathcal{F}_{\mathbb{Z}} T_{\sigma} f, f\right)_{\ell^{2}\left(\mathbb{Z}\right)} & \geq C_{0}\|f\|_{\ell_{m,\Lambda}^{2}\left(\mathbb{Z}\right)}^{2}-C_{1}\|f\|_{\ell^{2}\left(\mathbb{Z}\right)}^{2}, \quad \text { since } \mathcal{F}_{\mathbb{Z}}^{*} \mathcal{F}_{\mathbb{Z}}=I d, \\
			\operatorname{Re}(T_{\sigma} f, f)_{\ell^{2}\left(\mathbb{Z}\right)} & \geq C_{0}\|f\|_{\ell_{m,\Lambda}^{2}\left(\mathbb{Z}\right)}^{2}-C_{1}\|f\|_{\ell^{2}\left(\mathbb{Z}\right)}^{2},
		\end{aligned}
		$$
		and this completes the proof.
	\end{pff}
	
	We now proceed to prove the sharp Gårding inequality for $M$-elliptic operators on $\mathbb{T}$ and $\mathbb{Z}$, respectively. For this, first we will state the sharp G\r{a}rding's inequality for $M$-elliptic operators on $\mathbb{T}$ without proof as it follows similar lines to the proof of \cite[Corollary 5.9]{AB&GK&MK}.
	\begin{theorem} \textbf{(Sharp G\r{a}rding's inequality for $M$-elliptic operators on $\mathbb{T}$)}\label{Sharp gardings theorem on T}\\ 
		Let $\sigma \in  M_{\rho, \Lambda}^{m}\left(\mathbb{T}\times \mathbb{Z}\right)$ be such that $\sigma(x, k) \geq 0$, for all $(x,k) \in \mathbb{T} \times \mathbb{Z}$. Then there exists a positive constant $C$ such that for all $f \in H_{\Lambda}^{\frac{m-1}{2},2}\left(\mathbb{T}\right)$, we have
		$$
		\operatorname{Re}(T_{\sigma} f, f)_{L^{2}\left(\mathbb{T}\right)} \geq -C\|f\|_{H_{\Lambda}^{\frac{m-1}{2},2}\left(\mathbb{T}\right)}^{2}.
		$$
	\end{theorem}
	In the following result, we prove that Theorem \ref{Sharp gardings theorem on T} implies the corresponding sharp Gårding's inequality for $M$-elliptic operators on $\mathbb{Z}$.
	\begin{theorem} \textbf{(Sharp G\r{a}rding's inequality for $M$-elliptic operators on $\mathbb{Z}$)}\label{Sharp gardings theorem on Z}\\ 
		Let $\sigma \in  M_{\rho, \Lambda}^{m}\left(\mathbb{Z}\times \mathbb{T}\right)$ be such that $\sigma(k, x) \geq 0$ for all $(k,x) \in \mathbb{Z} \times \mathbb{T}$.  Then there exists a positive constant $C$ such that for all $f \in \ell_{\frac{m-1}{2},\Lambda}^{2}\left(\mathbb{Z}\right)$, we have
		$$
		\operatorname{Re}(T_{\sigma} f, f)_{\ell^{2}\left(\mathbb{Z}\right)} \geq -C\|f\|_{\ell_{\frac{m-1}{2},\Lambda}^{2}\left(\mathbb{Z}\right)} .
		$$
	\end{theorem}
	\begin{pff}
		Let $\tau(x, k)=\overline{\sigma(-k, x)}$. Then by Theorem  \ref{relation between tau and sigma}, we have
		$$
		T_{\sigma}=\mathcal{F}_{\mathbb{Z}}^{-1} T_{\tau}^{*} \mathcal{F}_{\mathbb{Z}}.
		$$
		Using the same argument and notation as in the proof of Theorem \ref{Gardings theorem on Z}, and by Theorem \ref{Sharp gardings theorem on T}, we get
		$$
		\begin{aligned}
			\operatorname{Re}(T_{\sigma} f, f)_{\ell^{2}\left(\mathbb{Z}\right)} &=\operatorname{Re}\left(\mathcal{F}_{\mathbb{Z}}^{-1}T_{\tau}^{*} g, \mathcal{F}_{\mathbb{Z}}^{-1} g\right)_{\ell^{2}\left(\mathbb{Z}\right)} \\
			&=\operatorname{Re}\left(T_{\tau}^{*} g, g\right)_{L^{2}\left(\mathbb{T}\right)} \\
			&=\operatorname{Re}\left(T_{\tau} g, g\right)_{L^{2}\left(\mathbb{T}\right)} \\
			& \geq-C\|g\|_{H_{\Lambda}^{\frac{m-1}{2},2}\left(\mathbb{T}\right)}^{2} \\
			&=-C\|f\|_{\ell_{\frac{m-1}{2},\Lambda}^{2}\left(\mathbb{Z}\right)},
		\end{aligned}
		$$
		and this completes the proof of the theorem.
	\end{pff}

	\section{Applications}\label{sec6}
	In this section, we present an application of Gårding's inequality for the   class $M_{\rho, \Lambda}^{m}\left(\mathbb{T} \times \mathbb{Z}\right).$ Let  $T_{\sigma, 0} $ and $T_{\sigma, 1}$  are the minimal and maximal pseudo differential operator of $T_\sigma$ on $L^2(\mathbb{T})$ defined as in   \cite{kal}. First, we recall the following definitions about strongly elliptic symbols and strong solutions.
	\begin{defi}
		Let $\sigma \in M_{\rho,\Lambda}^m\left(\mathbb{T} \times \mathbb{Z}\right),$ $m \in \mathbb{R}$. Then $\sigma$ is said to be strongly $M$-elliptic if there exist positive constants $C$ and $R$ for which
		$$\operatorname{Re}(\sigma(x,k)) \geq C\,\Lambda(k)^{m},\quad |k| \geq R.$$
	\end{defi}
	\begin{defi}
		Let $\sigma \in M_{\rho,\Lambda}^{2m}\left(\mathbb{T} \times \mathbb{Z}\right), m>0$, and let $f \in L^{p}\left(\mathbb{T}\right), 1<p<\infty$. A function $u \in L^{p}\left(\mathbb{T}\right)$ is   a strong solution of the equation $T_{\sigma} u=f$ if $u \in \mathcal{D}\left(T_{\sigma, 0}\right)$ and $T_{\sigma, 0} u=f$.
	\end{defi}
	Proceeding similarly as in Theorem 18.2 of \cite{MWbook}, we have the following result related to 
	the strong solution of the equation $T_{\sigma} u=f$ on $\mathbb{T}$.
	\begin{lemma}\label{Lemma}
		Let $\sigma \in M_{\rho,\Lambda}^{2m}\left(\mathbb{T} \times \mathbb{Z}\right), m>0$, be an $M$-elliptic symbol  such that 
		\begin{align}\label{garding app}
			\operatorname{Re}\left(T_{\sigma} \varphi, \varphi\right) \geq C \|\varphi\|_{m, 2, \Lambda}^{2}, \quad \varphi \in H_{\Lambda}^{m,2},
		\end{align}
		for some    positive  constant $C$.  
		Then for every function $f$ in $L^{2}\left(\mathbb{T}\right)$, the pseudo-differential equation $T_{\sigma} u=f$   has a unique strong solution $u$ in $L^{2}\left(\mathbb{T}\right)$.
	\end{lemma} 
	In the next result, we provide sufficient conditions for the existence and uniqueness of strong solutions in $L^{2}(\mathbb{T})$ for the pseudo-differential operator $T_{\sigma}$ with strongly elliptic symbol.
	\begin{theorem} Let $\sigma \in  M_{\rho,\Lambda}^{2m}\left(\mathbb{T} \times \mathbb{Z}\right), m>0$, be a strongly elliptic symbol. Then  for all $f$ in $L^{2}\left(\mathbb{T}\right)$  there exists a real number $\lambda_{0}$ such that  for   all $\lambda \geq \lambda_{0}$, the pseudo-differential equation $\left(T_{\sigma}+\lambda I\right) u=f$ on $\mathbb{T}$ has a unique strong solution $u$ in $L^{2}\left(\mathbb{T}\right)$, where $I$ is the identity operator on $L^{2}\left(\mathbb{T}\right)$.
	\end{theorem}
	
	\begin{pff}
		From  Gårding's inequality (\ref{garding}), there exist constants $A>0$ and $\lambda_{0}$ such that  
		$$
		\operatorname{Re}\left(T_{\sigma} \varphi, \varphi\right) \geq A\|\varphi\|_{m, 2, \Lambda}^{2}-\lambda_{0}\|\varphi\|_{0, 2, \Lambda}^{2}, \quad \varphi \in H_{\Lambda}^{m,2}.
		$$
		
		Now  for any $\lambda \geq \lambda_{0}$, we get 
		\begin{align*}
			\operatorname{Re}\left(\left(T_{\sigma}+\lambda I\right) \varphi, \varphi\right) &=\operatorname{Re}\left (T_{\sigma}   \varphi, \varphi\right)+\left(\lambda-\lambda_{0}\right)\|\varphi\|_{2}^{2}  -\lambda_{0}\|\varphi\|_{2}^{2}  \\
			&\geq A\|\varphi\|_{m, 2, \Lambda}^{2}+\left(\lambda-\lambda_{0}\right)\|\varphi\|_{2}^{2} \geq A\|\varphi\|_{m, 2, \Lambda}^{2}.
		\end{align*}
		This shows that  $	  T_{\sigma}$  satisfies the condition (\ref{garding app}).  Thus by  Lemma \ref{Lemma},   the pseudo-differential equation $\left(T_{\sigma}+\lambda I\right) u=f$  has a unique strong solution $u$ in $L^{2}\left(\mathbb{T}\right)$ and this completes the proof.
	\end{pff}

\end{document}